\documentclass[11pt]{article}
\usepackage{amsmath,amssymb,amsthm,eucal, mathrsfs, mathtools,hyperref}
\usepackage{xcolor}
\usepackage[all]{xy}

\title{Levinson's theorem as an index pairing}
\author{Angus Alexander, 
Adam Rennie\thanks{email: 
\texttt{angusa@uow.edu.au, renniea@uow.edu.au}}
\\[3pt]
School of Mathematics and 
Applied Statistics, University of Wollongong,\\
Wollongong, Australia\\
}

\advance\topmargin by -\headheight
\advance\topmargin by -\headsep
\evensidemargin=\oddsidemargin
\makeatletter
\def\section{\@startsection{section}{1}{\z@}{-3.5ex plus -1ex minus
  -.2ex}{2.3ex plus .2ex}{\large\bf}}
\def\subsection{\@startsection{subsection}{2}{\z@}{-3.25ex plus -1ex
  minus -.2ex}{1.5ex plus .2ex}{\normalsize\bf}}
\makeatother

\numberwithin{equation}{section} %% needs `amsmath' package

\theoremstyle{plain} %% needs `amsmath' package
\newtheorem{thm}{Theorem}[section]
\newtheorem{lemma}[thm]{Lemma}

\newtheorem{cor}[thm]{Corollary}

\theoremstyle{definition} %% needs `amsmath' package
\newtheorem{defn}[thm]{Definition}
\newtheorem{ass}[thm]{Assumption}
\theoremstyle{remark} %% needs `amsmath' package
\newtheorem{rmk}[thm]{Remark}

\DeclareMathOperator{\Dom}{Dom}   %% domain of an operator
\newcommand{\eps}{\varepsilon} %% short for \varepsilon
\newcommand{\A}{\mathcal{A}}  %% an algebra
\newcommand{\B}{\mathcal{B}}  %% another algebra
\newcommand{\C}{\mathbb{C}}   %% complex numbers
\newcommand{\D}{\mathcal{D}}  %% a self-adjoint operator
\renewcommand{\d}{\mathrm{d}} %% \d a = [\D,a]

\newcommand{\e}{\mathrm{e}}
\newcommand{\F}{\mathcal{F}}  %% another bi/module
\newcommand{\Fi}{\mathcal{F}} %% or Fourier transform
\renewcommand{\H}{\mathcal{H}}  %% a Hilbert space
\newcommand{\K}{\mathcal{K}}  %% compact endomorphisms
\newcommand{\N}{\mathbb{N}}   %% natural numbers
\newcommand{\norm}[1]{\left|\left|#1\right|\right|} %% for norms
\newcommand{\ox}{\otimes}     %% tensor product
\newcommand{\R}{\mathbb{R}}   %% real numbers
\newcommand{\Sf}{\mathbb{S}}  %% sphere
\newcommand{\vp}{\varphi}
\newcommand{\pd}[2]{\frac{\partial#1}{\partial#2}} %% partial deriv
\newcommand{\stroke}{\mathbin|}     %% (for `\pair' and such)
\newcommand{\od}[2]{\frac{\mathrm{d}#1}{\mathrm{d}#2}}

\def\pairL_#1(#2|#3){{}_{#1}(#2\stroke#3)} %% hermitian pairing _B(s|t)
\def\pairR(#1|#2)_#3{(#1\stroke#2)_{#3}} %% hermitian pairing (s|t)_A
\def\scal<#1|#2>{\langle#1\stroke#2\rangle} %% scalar product <y|z>

\renewcommand{\epsilon}{\varepsilon}

\theoremstyle{definition}

\definecolor{MyBlue}{cmyk}{1,0.13,0,0.63}
\definecolor{MyGreen}{cmyk}{0.91,0,0.88,0.52}
\newcommand{\mylinkcolor}{MyBlue}
\newcommand{\mycitecolor}{MyGreen}
\newcommand{\myurlcolor}{black}

\usepackage{hyperref}
\hypersetup{%
  bookmarksnumbered=true,bookmarksopen=false,%
  plainpages=false,% necessary to prevent duplicate page identifiers
  linktocpage=true,%
  colorlinks=true,breaklinks=true,%
  linkcolor=\mylinkcolor,citecolor=\mycitecolor,urlcolor=\myurlcolor,%
  pdfpagelayout=OneColumn,%
  pageanchor=true,%
}

\begin{document}

\maketitle

\vspace{-2pc}

\begin{abstract}
We build on work of Kellendonk, Richard, Tiedra de Aldecoa and others to show that the wave operators for Schr\"{o}dinger scattering theory on $\R^n$ generically have a particular form. As a consequence, Levinson's theorem can be interpreted as the pairing of the $K$-theory class of the unitary scattering operator and the $K$-homology class of the generator of the group of dilations. 
\end{abstract}
\maketitle

\parindent=0.0in
\parskip=0.00in

%\tableofcontents

\parskip=0.06in

\section{Introduction}
\label{sec:intro}

Levinson's theorem gives the number $N$ of bound states
of the Schr\"{o}dinger operator $-\Delta+V$, for suitably decaying potential $V$, as the total phase shift \cite{levinson49}
\[
N=\frac{1}{\pi}(\delta(0)-\delta(\infty)).
\]
In this paper, for suitable potentials, we realise the number of bound states $N$ as the pairing between the $K$-homology class $[D_+]\in K^1(C_0(\R^+)\ox\K)$ of the generator of dilations $D_+$ on the half-line and the $K$-theory class $[S]\in K_1(C_0(\R^+)\ox\K)$ of the scattering operator $S$. In this language the number of bound states is  given by
\[
\langle[D_+],[S]\rangle=N.
\]

The topological interpretation of Levinson's theorem
was initiated in several low dimensional cases in \cite{kellendonk06, kellendonk08, kellendonk12}. 
The technique in \cite{kellendonk06, kellendonk08, kellendonk12} is to use the form of the wave operator to deduce that $W_-$ belongs to an algebra $E$ fitting in to an exact sequence, with $\K$ the compact operators, 
\[
0\to \K\to E\to C(\Sf^1)\to 0.
\] 
As a consequence one may directly compute that the $K$-theory boundary map (the Index map) $K_1(C(\Sf^1))\to K_0(\mathcal{K})$ takes the scattering operator to the index of the wave operator, which is known to be $-N$.

Realising $W_-$ as an element of the algebra $E$ also allows another topological interpretation, which is the starting point of this paper. 
Following \cite{kellendonk06, kellendonk08, kellendonk12,richard13},  we show that for suitably decaying potentials, the wave operators for Schr\"{o}dinger scattering on Euclidean space are generically of the form
\begin{align}
\label{eq:waveopform}
W_- &= \textup{Id}+\vp(D_n)(S-\textup{Id})+K.
\end{align}
Here $K$ is a compact operator, $D_n$ denotes the generator of the dilation group on $L^2(\R^n)$, $S = W_+^*W_-$ the unitary scattering operator, and $\vp: \R \to \C$ is given by 
\[
\vp(x) = \frac12 (1+\tanh{(\pi x)} - i \cosh{(\pi x)}^{-1}).
\] 
This form for the wave operator allows us to show that the number of bound states is equal to the index pairing between the $K$-theory class of the scattering operator and the $K$-homology class of the generator of dilations on the half-line. This index pairing requires care, as the algebras involved are nonunital, and we add the required detail in Section \ref{sec:pairing}.

That the wave operator has the form of Equation  \eqref{eq:waveopform} has been shown in numerous low-dimensional Euclidean examples by \cite{kellendonk06, kellendonk08, kellendonk12, richard13, richard13ii, richard21}. Similar formulae for the wave operators have been developed in various other scattering models including rank-one perturbations \cite{richard10}, the Friedrichs-Faddeev model \cite{isozaki12}, Aharonov-Bohm operators \cite{kellendonk11}, lattice scattering \cite{bellissard12}, half-line scattering \cite{inoue20}, discrete scattering \cite{inoue19}, scattering for an inverse-square potential \cite{inoue19ii} and for 1D Dirac operators with zero-range interactions \cite{richard14}.

%While it is intrinsically fascinating that the wave operators generically have the universal form of Equation \eqref{eq:waveopform}, it has been shown in numerous cases \cite{kellendonk06, kellendonk08, kellendonk12} that this form allows one to realise Levinson's theorem, a well known result from scattering theory relating the number of eigenvalues of a Schr\"{o}dinger operator to the total phase shift \cite{levinson49}, as a topological statement. 

%The technique in \cite{kellendonk06, kellendonk08, kellendonk12} is to use the form of the wave operator to deduce that $W_-$ belongs to an algebra $E$ fitting in to an exact sequence
%\[
%0\to \K\to E\to C(\Sf^1)\to 0.
%\] 
%As a consequence one may deduce that the boundary map (the Index map) $K_1(\Sf^1)\to K_0(\mathcal{K})$ takes the scattering operator to the index of the wave operator.

%Our approach to the topological interpretation of Levinson's theorem is the pairing between $K$-theory and $K$-homology. The form of the wave operator suggests that the number of bound states is the index pairing between the dilation on the half-line, and the scattering operator, \cite{HR}. 

As well as the formula for the wave operator, which depends on having potentials with reasonable decay, we require that $S(0)=\textup{Id}$. This assumption rules out dimension 1 quite generally, and moreover the formula for the wave operator is slightly different in this case, due to the special properties of the zero-sphere $\Sf^0$ \cite{kellendonk08}. The only other case ruled out by the assumption $S(0)=\textup{Id}$ is when there is a zero energy resonance in dimension 3. 

Two other cases are ruled out by the intransigence of the low energy asymptotic expansion of the resolvent. These are the two dimensional case in the presence of zero energy $p$-resonances \cite{richard13ii, richard21} and the case of resonances in dimension $n = 4$ (see Lemma \ref{lem:mult}).

One sad feature of our result is that the one case (of dimension $n=1$) where we would have the summability to apply the local index formula \cite{carey06} is explicitly ruled out. Trace class estimates of \cite[Section 8.1]{yafaev10} suggest that the scattering operator is sensitive to such summability data, and so the use of the local index formula in scattering theory is of interest for further study. Whilst we cannot directly use the local index formula, there are many analytic formulae for computing the number of bound states using scattering data, see \cite{bolle77} in dimension $3$, \cite[Theorem 6.3]{bolle88} in dimension $2$ and \cite[Theorem 3.66]{zworski19} for higher (odd) dimensions.

The layout of the paper is as follows. In Section \ref{sec:prelim} we introduce the relevant concepts from scattering theory and fix our notation. In particular we discuss resolvent expansions and how they can be used to analyse the scattering matrix. In Section \ref{sec:wave-op} we prove, via a number of technical results, that the wave operator in dimension $n=2$ and $n \geq 4$ is of the form of Equation \eqref{eq:waveopform}, with dimension $n=3$ having been considered already in \cite{richard13}. In Section \ref{sec:pairing} we use the universal form of the wave operator to reinterpret Levinson's theorem as an index pairing between the $K$-theory class of the scattering operator $S$ and the $K$-homology class of the generator of dilations on the half-line $D_+$.

{\bf Acknowledgements:} 
Both authors would like to thank Johannes Kellendonk for his hospitality and many useful discussions during the visit of the first author to Lyon. Both authors gratefully acknowledge numerous enlightening conversations with Serge Richard. Both authors thank Alan Carey for a careful reading of draft versions. The first author also acknowledges the support of an Australian Government Research Training Program (RTP) Scholarship. This project was supported by the ARC Discovery grant DP220101196.

%%%%%%%%%%%%%%%%%%%%%%%%%%%%%%%%%%%%%%%%%%%%%%%%%%%%%%%%%%%%%%%%%%%%%
%%%%%%%%%%%%%%%%%%%%%%%%%%%%%%%%%%%%%%%%%%%%%%%%%%%%%%%%%%%%%%%%%%%%%
%%%%%%%%%%%%%%%%%%%%%%%%%%%%%%%%%%%%%%%%%%%%%%%%%%%%%%%%%%%%%%%%%%%%%

\section{Preliminaries on scattering theory} \label{sec:prelim}
\label{sec:scats}

\subsection{Standing assumptions and notation}
\label{subsec:ass-notes}

Throughout this document we will consider the scattering theory on $\R^n$ associated to the operators
\[
H_0=-\sum_{j=1}^n\frac{\partial^2}{\partial x_j^2}=-\Delta
\quad\mbox{and}\quad
H=-\sum_{j=1}^n\frac{\partial^2}{\partial x_j^2}+V
\]
where the (multiplication operator by the) 
potential $V$ is real-valued and satisfies 
\begin{align}
\label{ass11}
|V(x)| &\leq C(1+|x|)^{-\rho}
\end{align}
for various values of $\rho$ depending on the particular result and the dimension $n$. 
We denote the Schwartz space $\mathcal{S}(\R^n)$ and its dual $\mathcal{S}'(\R^n)$ and recall the weighted Sobolev spaces 
\[
H^{s,t}(\R^n)= \Big\{f \in \mathcal{S}'(\R^n): \norm{f}_{H^{s,t}} := \norm{(1+|x|^2)^{\frac{t}{2}} (\textup{Id}-\Delta)^\frac{s}{2} f} < \infty \Big\}
\] 
with index $s \in \R$ indicating derivatives and $t \in \R$ associated to decay at infinity \cite[Section 4.1]{amrein96}. With $\langle\cdot,\cdot\rangle$ the Euclidean inner product on $\R^n$, we denote the Fourier transform by
\[
\F_n:L^2(\R^n)\to L^2(\R^n),\qquad [\F_n f](\xi)=(2\pi)^{-\frac{n}{2}}\int_{\R^n}e^{-i\langle x,\xi\rangle}f(x)\,\d x.
\]
Note that the Fourier transform $\F_n$ is an isomorphism from $H^{s,t}$ to $H^{t,s}$ for any $s,t \in \R$. We will frequently drop the reference to the space $\R^n$ for simplicity of notation, as well as the subscript $n$ from our Fourier transforms when the dimension is clear. We denote by $\mathcal{B}(\H_1,\H_2)$ and $\mathcal{K}(\H_1,\H_2)$ the bounded and compact operators from $\H_1$ to $\H_2$.
For $z \in \C \setminus \R$, we let
\[
R_0(z)=(H_0-z)^{-1},\qquad R(z)=(H-z)^{-1}
\]
and the boundary values of the resolvent are defined as
\begin{align}
R_0(\lambda \pm i0) &= \lim_{\eps \to 0}{R_0(\lambda \pm i \eps)}\quad\mbox{and}\quad R(\lambda \pm i0) = \lim_{\eps \to 0}{R(\lambda \pm i \eps)}.
\end{align}
The limiting absorption principle \cite[Theorem 5.3.7]{kuroda78} tells us that these boundary values exist in $\mathcal{B}(H^{-1,t},H^{1,-t})$ for any $t > \frac12$ and $\lambda \in (0,\infty)$. The operator $H_0$ has purely absolutely continuous spectrum, and in particular no kernel. The operator $H$ can have eigenvalues and for $V$ satisfying Assumption \eqref{ass11} with $\rho > 1$ we have that these eigenvalues are negative, or zero \cite[Theorem 6.1.1]{yafaev10}. We let $P_0$ be the kernel projection of $H$, which may be zero.

Much of the analysis of the asymptotics of $R(z)$ is simplified by studying the operator (whose two expressions are related  by resolvent identities)
\begin{equation}
T(z) = V(\textup{Id} +R_0(z)V)^{-1}=V-VR(z)V.
\label{eq:tee}
\end{equation}
The operator $T(z)$  appears in our analysis of the wave and scattering operators. The operator $T(z)$ also appears in the scattering amplitude (the integral kernel of $S-\textup{Id}$, see Theorem \ref{thm: stationary scattering operator}).

The one-parameter unitary group of dilations on $L^2(\R^n)$ is
given on $f\in L^2(\R^n)$ by
\begin{align}\label{defn:dilation}
[U_n(t)f](x) &= \e^{\frac{nt}{2}} f(\e^t x),\qquad t\in\R.
\end{align}
We denote the self-adjoint generator of $U_n$ by $D_n$.
The generator of the group $(U_+(t))$ of dilations on the half-line $\R^+$ is denoted $D_+$
(which is $D_1$ restricted to the positive half-line). The generators of the dilation groups are given by
\begin{align}
D_+ &= \frac{y}{i} \od{}{y} + \frac{1}{2i}{\rm Id},\qquad D_n=\sum_{j=1}^n\frac{x_j}{i}\frac{\partial}{\partial x_j}+\frac{n}{2i}{\rm Id}.
\end{align}
Since each of $D_+,D_n$ generate one-parameter groups, we can recognise functions of these operators. For $D_+$ and $\vp:\R\to\C$ a bounded function whose Fourier transform has rapid decay, we have
\begin{align*}
[\vp(D_+) g](\rho) &= (2\pi)^{-\frac12} \int_\R{[\Fi^* \vp](t) \e^{\frac{t}{2}} g(\e^t \rho)\, \d t},
\end{align*}
with a similar formula for $D_n$.

Several Hilbert spaces recur, and we adopt the notation (following \cite[Section 2]{jensen81} which contains an excellent discussion on the relations between the spaces and operators we introduce here)
\[
\H = L^2(\R^n),\quad \mathcal{P} = L^2(\Sf^{n-1}),\quad \H_{spec} = L^2(\R^+, \mathcal{P}) \cong L^2(\R^+) \otimes \mathcal{P}.
\] 
Here $\H_{spec}$ provides the Hilbert space on which we can diagonalise the free Hamiltonian $H_0$.

Since $V$ is bounded, $H = H_0+V$ is self-adjoint with $\Dom(H) = \Dom(H_0)$. The wave operators 
\[
W_\pm=s-\lim_{t\to\pm\infty}e^{itH}e^{-itH_0}
\]
exist and are asymptotically complete if $\rho > 1$ \cite[Theorem 12.1]{pearson88}. We will use the stationary scattering theory, which coincides with the time-dependent approach \cite[Section 5.3]{yafaev92} given our assumptions.
For suitable $f,g \in \H$ we can write \cite[Equation 0.6.9]{yafaev10}
\begin{align}\label{eq:statwaveop}
\langle W_\pm f, g \rangle &= \int_\R{\left(\lim_{\eps \to 0}{\frac{\eps}{\pi}\langle R_0(\lambda \pm i \eps) f, R(\lambda \pm i \eps) g \rangle } \right)\, \d \lambda}.
\end{align}

%\subsection{Diagonalisations and wave operators}
%\label{subsec:diag}

For our analysis of the wave operator, we must describe the explicit unitaries which diagonalise our Hamiltonians. 

For $\lambda > 0$ the trace operator $\gamma(\lambda) : \mathcal{S}(\R^n) \to \mathcal{P}$ defined by
$[\gamma(\lambda) f](\omega):= f(\lambda \omega)$ defines a bounded operator and for each $s > \frac12$ and $t \in \R$ extends to a bounded operator on $H^{s,t}$ (see \cite[Theorem 2.4.3]{kuroda78}).

\begin{defn}
\label{def:diag}
For  $\lambda \in \R^+$, $s \in \R$ and $t > \frac12$ we define the operator 
\[
\Gamma_0(\lambda): H^{s,t} \to \mathcal{P}\quad \mbox{by}\quad
[\Gamma_0(\lambda) f](\omega) = 2^{-\frac12} \lambda^{\frac{n-2}{4}} [\F f](\lambda^\frac12 \omega)
\]
and the operator which diagonalises the free Hamiltonian $H_0$ as
\[
F_0: \H \to \H_{spec}\quad \mbox{by} \quad [F_0 f](\lambda,\omega) = [\Gamma_0(\lambda) f](\omega).
\] 
\end{defn}

\begin{lemma}[{\cite[p. 439]{jensen81}}]
\label{lem:eff-emm}
The operator $F_0$ is unitary. Moreover for $\lambda \in[0,\infty)$, $\omega\in \Sf^{n-1}$ and $f \in \H_{spec}$ we have
\[
[F_0H_0F_0^* f](\lambda,\omega)=\lambda f(\lambda,\omega)=:[Mf](\lambda,\omega).
\] 
\end{lemma}
Here we have defined the operator $M$ of multiplication by the spectral variable.

\subsection{Resolvent expansions and resonances}\label{sec:resolvents}

Here we recall some known results regarding expansions related to the perturbed resolvent $R(z)$ in the limit $z \to 0$. Only the terms in the expansion relevant to later computations will be shown, however we note that higher terms can be computed explicitly \cite{jensen79, jensen80, jensen84}. The low energy behaviour is sensitive to the presence of `zero-energy resonances'. These are essentially distributional solutions to $H\psi = 0$ which are not square-integrable but lie in some larger space. We will give the precise definition shortly.

The operator of interest to us is not $R(z)$ but the related operator $T(z) = V(\textup{Id} +R_0(z)V)^{-1}$ from \eqref{eq:tee}. Before giving the low-energy expansions of $T(z)$, we note that the high-energy behaviour can be described rather simply in the following generalisation of \cite[Lemma 9.1]{jensen79}.

\begin{lemma}
\label{lem:resolvents}
Suppose that $\rho > \frac{n+1}{2}$ in Assumption \eqref{ass11} and define the complex domain $C_+ = \{z \in \C: \textup{Re}(z) \geq 1 \textup{ and } \textup{Im}(z) > 0 \}$ and $t \in (\frac12,\rho-\frac12)$. Then $VR_0(z)$ and $(\textup{Id} +VR_0(z))^{-1}$ can be extended to continuous and uniformly bounded functions from $C_+$ to $\mathcal{B}(H^{0,t},H^{0,t})$. Similarly, $R_0(z)V$ and $(\textup{Id} +R_0(z)V)^{-1}$ can be extended to continuous and uniformly bounded functions from $C_+$ to $\mathcal{B}(H^{0,-t},H^{0,-t})$.
\end{lemma}
\begin{proof}
Since for any  $\rho> \frac{n+1}{2}$ and  $t \in (\frac12, \rho-\frac12)$ we have $R_0(z) \in \mathcal{B}(H^{0,t},H^{0,t-\rho})$  is continuous in $z \in C_+$ by \cite[proof of Lemma 3.1]{jensen80}, and since $V \in \mathcal{B}(H^{0,t-\rho},H^{0,t})$ by assumption, we find that $VR_0(z) \in \mathcal{B}(H^{0,t},H^{0,t})$ for any $t \in (\frac12, \rho-\frac12)$.

Since $VR_0(z)$ is compact and has no eigenvalue $-1$ for $z \in C_+$ (see \cite[Proposition 5.2.1]{kuroda78}), the operator $(\textup{Id} +VR_0(z))^{-1} \in \mathcal{B}(H^{0,t},H^{0,t})$ exists and is continuous in $z \in C_+$. By \cite[Theorem 1]{murata84} we have that $VR_0(z) \to 0$ as $|z| \to \infty$ in $C_+$ in the norm of $\mathcal{B}(H^{0,t}, H^{0,t})$ for any $t > \frac12$, so that the operator $(\textup{Id}+VR_0(z))^{-1}$ is uniformly bounded. A similar computation (or duality) can be used to prove the second claim.
\end{proof}

\subsubsection{Resonances in dimensions 1,2,3}

The formulae for wave operators in terms of dilations and scattering operator that we seek to prove have been established in dimensions $n=1,2,3$, \cite{kellendonk06, kellendonk08, kellendonk12, richard13, richard13ii, richard21}.
For this reason we do not need to describe the resolvent expansions in these cases, and refer the reader to \cite{jensen01} for relevant expansions in dimensions $n=1,2$ and \cite{jensen79} for dimension $n=3$. The important takeaways from the low dimensional expansions are the following definitions of zero energy resonances.

\begin{defn}
\begin{enumerate}
\item If $n = 1$ we say there is a resonance at zero energy if there exists a $0 \neq \psi \in L^\infty(\R)$ such that $H\psi = 0$ in the sense of distributions. 
\item If $n = 2$ we say there is an $s$-resonance at zero energy if there exists a $0 \neq \psi \in L^\infty(\R^2)$ with $\psi \notin L^q(\R^2)$ for all $q < \infty$ and such that $H\psi = 0$ in the sense of distributions. 
\item If $n = 2$ we say there is a $p$-resonance at zero energy if for some $q > 2$ there exists a $0 \neq \psi \in L^q(\R^2) \cap L^\infty(\R^2)$ such that $H\psi = 0$ in the sense of distributions.
\item If $n = 3$ we say there is a resonance at zero energy of there exists $\psi \notin L^2(\R^3)$ such that $H\psi = 0$ in the sense of distributions. 
\end{enumerate}
\end{defn}

\begin{rmk}
The notation of $s$-resonances and $p$-resonances is motivated by the usual notation for angular momentum modes, although this intuition seems only applicable to rotationally invariant potentials. Here the $s$-resonances correspond to angular momentum $\ell = 0$, and the $p$-resonances appear as moments in coordinate directions in an analogous manner to states of angular momentum $\ell = 1$ (see \cite[Theorem 6.2]{jensen01}). In dimension $n=3$ the exact form of resonances can be described explicitly, see \cite[Proposition 7.4.10]{yafaev10} and \cite[Lemma 3.22]{zworski19}. 
\end{rmk}

\subsubsection{Resonances and resolvent in dimension 4}

The expansion of the operator $T(z)$ in the case $n = 4$ is covered in \cite[Lemmas 4.1, 4.3, 4.5-4.6]{jensen84}. The expansion depends heavily on whether there exist zero energy eigenvalues and resonances, and we review some details here briefly. We can characterise the dependence of $T(z)$ on resonant and kernel behaviour by considering the coefficient operators arising from a low energy expansion of the free resolvent. Suppose $\rho > 12$ and let $m,t > 0$ and $m+t>2$ and define the operator $G_0 \in \mathcal{B}(H^{-1,t},H^{1,m})$ by
\begin{align*}
[G_0 f](x) &= (4\pi^2)^{-1} \int_{\R^4}{|x-y|^{-2} f(y)\, \d y}.
\end{align*}
For $t \in (0,\rho)$ we can thus consider the idempotent $Q \in \mathcal{B}(H^{1,-t}, H^{1,-t})$ onto the kernel of $\textup{Id}+G_0V$ in $H^{1,-t}$. For $t \in (0, \rho -2)$ we define the idempotent $Q_1 = (\textup{Id}-P_0 V G_1^0V)Q$ (that $Q_1$ is an idempotent follows from \cite[Lemma 3.6]{jensen84}), where for $m,t \geq 2$ the operator $G_1^0 \in \mathcal{B}(H^{-1,t},H^{-1,m})$ is defined by
\begin{align*}
[G_1^0 f](x) &= (4\pi^2)^{-1}\int_{\R^4}{\left(1-2\gamma+i\pi-2\ln{\frac{|x-y|}{2}} \right) g(y)\, \d y}.
\end{align*}
Here $\gamma$ denotes Euler's constant. By \cite[Lemma 3.3]{jensen84} the space $Q_1 H^{1,-t}$ is at most one-dimensional and thus we make the following definition.

\begin{defn}\label{defn:res4d}
Let $n = 4$. We say that there exists a zero-energy resonance if $Q_1 \neq 0$ so that $\dim{\rm Image}(Q_1)=1$, and in this case we fix a resonance function $\psi\in {\rm Image}(Q_1)$ by the normalisation
\begin{align*}
\norm{V\psi}_1 &= 4\pi.
\end{align*}
\end{defn}
We note that by \cite[Lemma 3.3]{jensen84} we have $V\psi \in L^1(\R^4)$.
For $z$ sufficiently small we have $a-\ln(z)$ is invertible, where
\begin{align}\label{eq:constantdefn}
a &:= 1-2\gamma+i\pi - (4\pi)^{-2} \int_{\R^4}{\int_{\R^4}{\ln{\left(\frac{|x-y|}{2}\right)}[V\psi](x) \overline{[V\psi ](y)}\, \d x}\, \d y}.
\end{align}
Note that by \cite[Lemma 3.10]{jensen84} the integral is well-defined.

Recalling that $P_0$ is the projection onto the kernel of $H$, we now have all the necessary ingredients to state our resolvent expansion in dimension $n=4$.

\begin{lemma}[{\cite[Lemmas 4.1, 4.3, 4.5, 4.6]{jensen84}}]\label{lem:res4d}
Let $n = 4$ and suppose that $\rho > 12$ in Assumption \eqref{ass11}. Let $t \in (6,\rho-6)$ and $\psi$ a normalised resonance function of Definition \ref{defn:res4d}. Then there exists $C \in \mathcal{B}(H^{1,-t}, H^{1,-t})$ such that we have the expansion
\begin{align*}
T(z) &= z^{-1} VP_0 V +z^{-1} (a-\ln{(z)})^{-1} \langle V\psi,\cdot \rangle V\psi - \ln{(z)} C+O(1)
\end{align*}
in $\mathcal{B}(H^{1,-t}, H^{1,-t})$ as $z \to 0$.
\end{lemma}

We note that in the case $P_0 = 0$ and $\psi = 0$, the decay assumptions on the potential $V$ can be relaxed significantly \cite[Lemma 4.1]{jensen84}.

\subsubsection{Resolvent expansions in dimensions 5 and higher}

In dimensions $n \geq 5$ there can be no resonant behaviour with the only poles of $T(z)$ at $z = 0$ arising from the non-triviality of the kernel of $H$ \cite{jensen80}. The integral kernel of the free resolvent $R_0(z)$ can be constructed explicitly in terms of Hankel functions \cite[Equation 3.1]{jensen80}. Resolvent expansions in terms of these Hankel functions are typically split into odd and even cases due to the logarithmic behaviour of Hankel functions in the even case, however for our purposes we only need the lowest order term which agrees in both the even and odd case.

In order to discuss the resolvent expansions in general, we record the decay assumptions on the potential required in each dimension.

\begin{ass}
\label{ass:best-ass}
In dimension $n$ we fix $\rho$ and $t$ such that
\begin{enumerate}
	\item if $n=2$ then $\rho > 11$;
	\item if $n=3$ then $\rho > 5$ and $t \in (\frac{5}{2}, \rho - \frac{5}{2})$ ;
	\item if $n=4$ then $\rho > 12$ and $t \in (6, \rho-6)$ ; and
	\item if $n \geq 5$ then $\rho > \frac{3n+4}{2}$ and $t \in \left(\frac{n}{2}, \rho-\frac{n}{2}\right)$.
\end{enumerate}
We assume that $|V(x)| \leq C(1+|x|)^{-\rho}$ for almost all $x \in \R^n$.
\end{ass}

We note that the requirements of Assumption \ref{ass:best-ass} are needed for the statement of Theorem \ref{thm:main} and can be relaxed to varying degrees in the intermediate results of Section \ref{sec:wave-op}. For simplicity of the statements we will consistently use Assumption \ref{ass:best-ass}.  The result of the expansion of $T(z)$ in dimension $n \geq 5$ is the following.

\begin{lemma}[{\cite[Lemmas 5.1, 5.3, 5.5 and 5.7]{jensen80}}]\label{lem:resexpbig}
Suppose that $n \geq 5$ and that $\rho, t$ and $V$ satisfy Assumption \ref{ass:best-ass}. 
%are such that 
%\begin{enumerate}
%	\item if $n = 5$ then $\rho > 7$ and $s \in \left(6-\frac{5}{2}, \rho - \left(6-\frac{5}{2}\right)\right)$;
%	\item if $n = 6$ then $\rho > 6$ and $s \in \left(3, \rho - 3\right)$;
%	\item if $n = 7, 8, 9, 10$ then $\rho > 6$ and $s \in \left(6-\frac{n}{2}, \rho - \left(6-\frac{n}{2}\right)\right)$;
%	\item if $n \geq 11$ is odd then $\rho > \frac{n+1}{2}$ and $s \in \left(6-\frac{n}{2}, \rho - \left(6-\frac{n}{2}\right)\right)$; and
%	\item if $n \geq 12$ is even then $\rho > \frac{n+1}{2}$ and $s \in \left(\frac12, \rho-\frac12\right)$.
%\end{enumerate}
Then we have the expansion
\begin{align*}
T(z) &= z^{-1} VP_0 V +o(1)
\end{align*}
in $\mathcal{B}(H^{1,-t}, H^{1,-t})$ as $z \to 0$.
\end{lemma}

Again we note that if the kernel projection $P_0 = 0$ then the decay assumptions in Lemma \ref{lem:resexpbig} can be weakened, \cite[Lemmas 5.1, 5.5]{jensen80}.
\subsection{The scattering operator and low energy expansions of the trace operator}
In this section we develop some properties of the trace operator\footnote{In the sense of restriction to a hypersurface.} at low energies which will be useful later, along with an application to the threshold behaviour of the scattering matrix. The scattering matrix and the scattering amplitude are related by the following result. To state the result, recall from Definition \ref{def:diag} and Lemma \ref{lem:eff-emm} the operator $\Gamma_0$.

\begin{thm}[{\cite[Theorem 6.6.10]{yafaev10}}]
\label{thm: stationary scattering operator}
Suppose that $V$ satisfies Assumption \eqref{ass11} for some $\rho > 1$. The scattering matrix $S(\lambda)$ is given for all $\lambda \in \R^+$ by the equation
\begin{align}\label{eq: scat matrix defn}
S(\lambda) &= \textup{Id} - 2\pi i \Gamma_0(\lambda)(V-VR(\lambda+i0) V)\Gamma_0(\lambda)^*.
\end{align}
For each $\lambda\in\R^+$, the operator $S(\lambda)$ is unitary in $\mathcal{P}=L^2(\Sf^{n-1})$ and depends continuously (in the sense of norm) on $\lambda \in \R^+$. 
\end{thm}

We note that  resolvent identities yield $T(z) = V(\textup{Id}+R_0(z)V)^{-1} = V-VR(z)V$. This identity and the stationary formula 
for the scattering matrix of Equation \eqref{eq: scat matrix defn}
demonstrate why we need to study the asymptotics of $T$.
The resolvent expansions of the previous section and Equation \eqref{eq: scat matrix defn} allow us to analyse the value of the scattering matrix at zero energy. 
To do so we also need the low energy behaviour of the operator $\Gamma_0$, which we now describe.

For $\lambda > 0$ the trace operator $\gamma(\lambda): \mathcal{S}(\R^n) \to \mathcal{P}$ is continuous and extends to a bounded operator in $\mathcal{B}(H^{s,t}, \mathcal{P})$ for each $s > \frac12$ and $t \in \R$ \cite[Theorem 2.4.3]{kuroda78}. We need the asymptotic development of $\gamma(\lambda^\frac12)\Fi$ as $\lambda \to 0$. We can compute for $f \in C_c^\infty(\R^n)$ the expansion
\begin{align*}
[\gamma(\lambda^\frac12) \Fi f](\omega) &= (2\pi)^{-\frac{n}{2}} \int_{\R^n}{\e^{-i\lambda^\frac12 \langle x,\omega \rangle} f(x)\, \d x} \\
&= (2\pi)^{-\frac{n}{2}} \int_{\R^n}{\sum_{j=0}^K{\frac{1}{j!}(i\lambda^\frac12)^j (-\langle x, \omega \rangle)^j f(x)}\, \d x}  + O\left(\lambda^{\frac{K+1}{2}}\right)\\
&=: \sum_{j=0}^K{(i\lambda^\frac12)^j [\gamma_j f](\omega)} + O\left(\lambda^{\frac{K+1}{2}}\right)
\end{align*}
as $\lambda \to 0$ in $\mathcal{B}(H^{s,t}, \mathcal{P})$ for appropriate $s,t$. The operators $\gamma_j$ can be considered as operators in certain weighted Sobolev spaces, with higher terms in the series requiring convergence in Sobolev spaces with higher decay. Jensen \cite[Equation 5.4]{jensen79} states the following result in dimension $n = 3$ and the result generalises in a straightforward manner to each dimension.
\begin{lemma}
\label{lem:gammabounded}
Fix $j \in \N$. For each $s \geq 0$ and $t > j + \frac{n}{2}$ we have $\gamma_j \in \mathcal{B}(H^{s,t}, \mathcal{P})$.
\end{lemma}
\begin{proof}
For $t >j+\frac{n}{2}$ and $s = 0$, we can estimate that 
\begin{align*}
\norm{\gamma_j f}_{\mathcal{P}}^2 &= \int_{\Sf^{n-1}}{|[\gamma_j f](\omega)|^2\, \d \omega} 
= \int_{\Sf^{n-1}}{\left| (2\pi)^{-\frac{n}{2}} \int_{\R^n}{(-\langle x, \omega \rangle)^j f(x)\, \d x} \right|^2\, \d \omega} \\
&\leq (2\pi)^{-n} \int_{\Sf^{n-1}}{\left( \int_{\R^n}{\left| \langle \omega, x \rangle \right|^j |f(x)|\, \d x} \right)^2\, \d \omega} \\
&= (2\pi)^{-n} \int_{\Sf^{n-1}}{\left( \int_{\R^n}{\left| \langle \omega, x \rangle \right|^j (1+|x|^2)^{-\frac{t}{2}} (1+|x|^2)^{\frac{t}{2}} |f(x)| \, \d x}\right)^2\, \d \omega}.
\end{align*}
We next use the Cauchy-Schwarz inequality to obtain  the estimate
\begin{align*}
&(2\pi)^{-n} \int_{\Sf^{n-1}}{\left( \int_{\R^n}{\left| \langle \omega, x \rangle \right|^j (1+|x|^2)^{-\frac{t}{2}} (1+|x|^2)^{\frac{t}{2}} |f(x)| \, \d x}\right)^2\, \d \omega} \\
&\leq (2\pi)^{-n} \int_{\Sf^{n-1}}{\left( \int_{\R^n}{|\langle \omega, x \rangle|^{2j}(1+|x|^2)^{-t}\, \d x}\right)^2\left(\int_{\R^n}{(1+|y|^2)^{t} |f(y)|^2 \, \d y}\right)^2\, \d \omega} \\
&= C_j \norm{f}_{H^{0,t}(\R^n)}^2.
\end{align*}
Thus for $t > j+\frac{n}{2}$ we find $\gamma_j \in \mathcal{B}(H^{0,t}, \mathcal{P})$. For $s > 0$ we use the inclusion $H^{s,t} \subset H^{0,t}$.
\end{proof}
For each $\lambda > 0$ the operator $\Gamma_0(\lambda): \mathcal{S}(\R^n) \to \mathcal{P}$ extends to an element of $\mathcal{B}(H^{s,t}, \mathcal{P})$ for each $s \in \R$ and $t > \frac12$, as well as $\lambda \mapsto \Gamma_0(\lambda) \in \mathcal{B}(H^{s,t}, \mathcal{P})$ is continuous. These follow immediately since $\Gamma_0(\lambda) = 2^{-\frac12} \lambda^{\frac{n-2}{4}} \gamma(\lambda^\frac12) \Fi$.
\begin{lemma}\label{lem:gammaexp}
For $s \geq 0$ and $t > \frac{n}{2}+1$ we have the expansion
\begin{align*}
\Gamma_0(\lambda) &= 2^{-\frac12} \lambda^{\frac{n-2}{4}}\gamma_0-2^{-\frac12} \lambda^{\frac{n}{4}} \gamma_1+o\left(\lambda^{\frac{n+2}{4}}\right)
\end{align*}
as $\lambda \to 0$ in $\mathcal{B}(H^{s,t},\mathcal{P})$.
\end{lemma}

Combining Lemma \ref{lem:gammaexp} with the results of the previous subsection we can determine the low energy behaviour of the scattering matrix. We first mention known results in dimensions $n=1,2,3$.

\subsubsection{The scattering matrix at zero energy in low dimensions}

The low energy behaviour of the scattering matrix has been determined using resolvent expansion in dimension $n = 1$ in  \cite{bolle85} (see also \cite[Theorem 2.15]{zworski19}, \cite{melgaard01} and \cite[Proposition 9]{kellendonk08}).
%, however the statement we present here can be found as \cite[Proposition 9]{kellendonk08}.
%
%\begin{thm}
%Let $n = 1$ and suppose $\rho > \frac{5}{2}$ in Assumption \ref{ass11}. Then the scattering matrix satisfies
%\begin{align*}
%S(0) &= \begin{cases} \begin{pmatrix} 0 & -1 \\ -1 & 0 \end{pmatrix}, \quad & \textup{ if there does not exist a resonance,} \\
%\begin{pmatrix} 2ab & a^2-b^2 \\ b^2-a^2 & 2ab \end{pmatrix}, \quad & \textup{ if there exists a resonance,}
%\end{cases}
%\end{align*}
%where $a, b \in \R$ and $a^2+b^2 = 1$.
%\end{thm}
%
%It should be mentioned that we have used a different basis to \cite[Proposition 9]{kellendonk08}, 
The essential feature to distinguish resonant behaviour in dimension 1 is that $\textup{det}S(0) = -1$ when there are no resonances and $\textup{det}S(0)= 1$ when there are resonances.
We note that the dimension 1 result is atypical in the sense that for all dimensions $n \geq 2$, if there does not exist a resonance at zero then $S(0) = \textup{Id}$. The generic failure of $S(0)=\textup{Id}$ means that our index pairing result does not apply to dimension 1.

The low energy behaviour of the scattering matrix in dimension $n = 2$ can be found in \cite[Theorem 1.1]{richard21}. We note that in contrast to dimension $n = 1$ the behaviour is independent of the presence of resonances.

\begin{thm}
Let $n = 2$ and suppose $\rho > 11$ in Assumption \eqref{ass11}. Then the scattering matrix satisfies $S(0) = \textup{Id}$.
\end{thm}

As in dimension $n = 1$, the behaviour of the scattering matrix at zero in dimension $n = 3$ is dependent on the existence of resonances and has been determined in \cite[Theorems 5.1-5.3]{jensen79}, which we state below.

\begin{thm}
Suppose that $\rho > 5$ in Assumption \eqref{ass11} and $n = 3$. Then we have
\begin{align*}
S(0) &= \textup{Id} - 2 P,
\end{align*}
where $P = 0$ if there are no resonances and $P$ is the projection onto the spherical harmonic subspace of order $0$ if there does exist a resonance.
\end{thm}

\subsubsection{The scattering matrix at zero energy in dimensions $n \geq 4$}

The next statement was almost certainly known to Jensen based on the comments in \cite{jensen80}.

\begin{thm}\label{thm:scatmatzero}
Suppose that $n \geq 4$ and that $\rho$ and $V$ satisfy Assumption \ref{ass:best-ass}. Then we have
\begin{align*}
S(0) &= \textup{Id}.
\end{align*}
\end{thm}
\begin{proof}
We recall the stationary formula for the scattering matrix from Equation \eqref{eq: scat matrix defn} and apply a resolvent identity to obtain
\begin{align*}
S(\lambda) &= \textup{Id}- 2\pi i \Gamma_0(\lambda) (V-VR(\lambda + i0)V) \Gamma_0(\lambda)^* \\
&= \textup{Id} - 2\pi i \Gamma_0(\lambda) V (\textup{Id}+R_0(\lambda+i0)V)^{-1} \Gamma_0(\lambda)^*\\
&= \textup{Id} - 2\pi i \Gamma_0(\lambda) T(\lambda+i0) \Gamma_0(\lambda)^*.
\end{align*}
As $\lambda \to 0$ we can then apply Lemma \ref{lem:gammaexp} to obtain the expansion for $\Gamma_0(\lambda)$ and $\Gamma_0(\lambda)^*$. 
%We recall $T(z) = V(\textup{Id}+R_0(z)V)^{-1}$ and analyse the behaviour of the scattering matrix near $\lambda = 0$ on a case by case basis
%using the results of Section \ref{sec:resolvents}.

For $n \geq 5$ we apply Lemma \ref{lem:resexpbig}, and the result is the expansion
\begin{align*}
\Gamma_0(\lambda) T(\lambda + i0) \Gamma_0(\lambda)^* &=  2^{-1} \lambda^{\frac{n-2}{2}} (\gamma_0-i \lambda^\frac12 \gamma_1 + o(\lambda))(VP_0 V \lambda^{-1} +o(1))(\gamma_0^* +i \lambda^\frac12 \gamma_1^* +o(\lambda)) \\
&\to 0 
\end{align*}
as $\lambda \to 0$ since $\frac{n-2}{2} > 1$.

In the case $n = 4$, we use Lemma \ref{lem:res4d} to obtain (with $\psi$ the normalised zero energy resonance of Definition \ref{defn:res4d}) the expansion 
\begin{align*}
&\Gamma_0(\lambda) T(\lambda + i0) \Gamma_0(\lambda)^*\\ 
&=  2^{-1} \lambda \Big(\gamma_0-i \lambda^\frac12 \gamma_1 + o(\lambda)\Big)\Big(VP_0 V \lambda^{-1} +\lambda^{-1}\Big(a-\ln(\lambda)-i\frac{\pi}{2}\Big)^{-1} \langle V\psi,\cdot \rangle V\psi+o(\ln(\lambda))\Big)  \\
&\qquad \times\Big(\gamma_0^* +i \lambda^\frac12 \gamma_1^* +o(\lambda)\Big) 
\to 0 
\end{align*}
as $\lambda \to 0$, since by \cite[Lemma 3.3]{jensen84} we have $P_0V\gamma_0^* = 0$.
\end{proof}

%%%%%%%%%%%%%%%%%%%%%%%%%%%%%%%%%%%%%%%%%%%%%%%%%%%%%%%%%%%%%%%%%%%%%
%%%%%%%%%%%%%%%%%%%%%%%%%%%%%%%%%%%%%%%%%%%%%%%%%%%%%%%%%%%%%%%%%%%%%
%%%%%%%%%%%%%%%%%%%%%%%%%%%%%%%%%%%%%%%%%%%%%%%%%%%%%%%%%%%%%%%%%%%%%

\section{The form of the wave operator}
\label{sec:wave-op}

We aim to show the wave operator is of the form
\begin{align}
W_- &= \textup{Id}+\frac12 \left(\textup{Id}+\tanh{(\pi D_n)} - i \cosh{(\pi D_n)}^{-1} \right)(S-\textup{Id})+K,
\end{align}
where $D_n$ is the generator of dilations on $L^2(\R^n)$, $S$ the scattering matrix and $K$ a compact operator. We will frequently use the notation $\vp(x) = \frac12 (1+\tanh{(\pi x)} - i\cosh{(\pi x)}^{-1})$. 
%In this section we derive explicit expressions for the wave operators in dimensions $n \geq 4$ based on the resolvent expansions of Section \ref{sec:resolvents}. In dimensions $n \geq 5$ there can be no resonances, making the analysis simpler. In dimension $n = 4$ we must exclude the possibility of resonances in our analysis. 
The main result of this section is the following theorem, whose proof relies on several preparatory results. %We again use the notation $\vp(x) = \frac12 (1+\tanh{(\pi x)}-i \cosh{(\pi x)}^{-1})$.

\begin{thm}
\label{thm:main}
Let $n\geq 2$ and suppose that $\rho$ satisfies Assumption \ref{ass:best-ass}, and let $V$ satisfy $|V(x)| \leq C (1+|x|)^{-\rho}$ for almost every $x \in \R^n$.  In dimension 2 we suppose also that there are no $p$-resonances and in dimension 4 we suppose there are no resonances.
%If $n = 2$, let $\rho > 11$. If $n = 3$, let $\rho > 5$. If $n = 4$ let $\rho > 12$ and suppose there are no resonances. If $n = 5$, let $\rho > 7$.  If $n = 6$, let $\rho > 6$. If $n  \geq 7$, let $\rho > n+1$. Let $V$ satisfy $|V(x)| \leq C (1+|x|)^{-\rho}$ for almost every $x \in \R^n$. 
Then  in $\mathcal{B}(\H)$ 
\begin{align*}
W_- &= \textup{Id}+\vp(D_n)(S-\textup{Id})+K,
\end{align*}
where $K \in \mathcal{K}(\H)$ and $\vp(x) = \frac12 (1+\tanh{(\pi x)}-i \cosh{(\pi x)}^{-1})$.
\end{thm}

We begin with the special case of dimension two, before setting the scene for the remainder of the proof. The main elements of the proof are: a careful analysis of the diagonalisation maps of Definition \ref{def:diag}; an analysis of the stationary (integral) formula for the wave operator using this information; an interchange of limits in the integral formula; finally an identification of a function of dilation and compact remainder. Our proof closely follows that of \cite{richard13}.

\subsection{Dimension 2}
%The expressions obtained here are based on formulas in one dimension \cite{kellendonk08}, two dimensions without resonances \cite{richard13ii} and with resonances \cite{richard21} and three dimensions \cite{richard13}. The method of proof follows closely that of \cite{richard13}, requiring only some dimension dependent changes.
%We note that the decay assumptions here may not be optimal. We begin with a discussion of known results in lower dimensions first.
%
%The statement of Theorem \ref{thm:main} is known in dimension $n=3$ \cite[Theorem 1.1]{richard13} and will be included in our later discussion. The formula for the wave operator differs slightly in dimension $n = 1$, and we refer to \cite{}.

%, with an extra prefactor on the term $\cosh{(\pi D_1)}^{-1}$.
%
%\begin{thm}
%Let $n=1$ and suppose $\rho > \frac{5}{2}$. Then we have
%\begin{align*}
%W_- &= \textup{Id}+\frac12 ( \textup{Id}+\tanh{(\pi D_1)} -i A \cosh{(\pi D_1)}^{-1})(S-\textup{Id})+K,
%\end{align*}
%where $K$ is compact and $[Af](x) = f(-x)$ denotes the antipodal map.
%\end{thm}

We first recall the known form of the wave operator in two dimensions in the presence of $s$-resonances. We then massage the formula from Theorem \ref{thm:2-dim-no-pee} to obtain the statement in Theorem \ref{thm:main}. The following is \cite[Theorem 1.3]{richard21}

\begin{thm}
\label{thm:2-dim-no-pee}
Let $n = 2$ and $V$ satisfy Assumption \eqref{ass11} for some $\rho > 11$ and suppose that there are no $p$-resonances. Then there exist two continuous functions $\eta,\tilde{\eta}:\R^+ \to \mathcal{K}(L^2(\Sf^1))$ vanishing at $0$ and $\infty$ and satisfying $\eta(H_0)+\tilde{\eta}(H_0) = S-\textup{Id}$, and such that
\begin{align*}
W_- &= \textup{Id} + \frac12\left(\textup{Id}+\tanh{\left(\frac{\pi D_2}{2}\right)}\right) \eta(H_0) +\frac12 \left(\textup{Id}+\tanh{(\pi D_2)}-i\cosh{(\pi D_2)}^{-1}\right) \tilde{\eta}(H_0)+K,
\end{align*}
with $K \in \mathcal{K}(\H)$.
\end{thm}
%
%The resonance behaviour plays a key role in the structure of the wave operator here, however we can recover our universal form rather simply. 

With the help of a technical result, we can recover our universal form of the wave operator from Theorem \ref{thm:2-dim-no-pee} rather simply. 
\begin{lemma}
\label{lem:f(x)g(nabla)}
Suppose that $f \in L^2(\R) \cap C(\R)$ and $g \in C_0(\R, \mathcal{K}(\mathcal{P}))$ is such that $g(\pm\infty) = 0$. Let $L$ denote the (densely defined) operator of multiplication by the variable in $L^2(\R^+)$. Then the operator $(f(D_+)\otimes \textup{Id}) g(\ln{(L)})$ defines a compact operator on $\H_{spec}$. Similarly if $h \in C(\R^+, \mathcal{K}(\mathcal{P}))$ is such that $h(0) = h(\infty) = 0$, then $(f(D_+) \otimes \textup{Id}) h(L)$ defines a compact operator on $\H_{spec}$.
\end{lemma}
\begin{proof}
The algebraic tensor product $C_0(\R) \odot \mathcal{K}(\mathcal{P})$ is dense in $C_0(\R, \mathcal{K}(\mathcal{P}))$ by \cite[Theorem 1.15]{prolla77}, when $C_0(\R, \mathcal{K}(\mathcal{P}))$ is equipped with the uniform topology. Thus it suffices to prove that for $a \in C_0(\R)$ and $b \in \mathcal{K}(\mathcal{P})$ the operator
\begin{align*}
(f(D_+) \otimes \textup{Id})(a(\ln{(L)}) \otimes b) &= f(D_+) a( \ln{(L)}) \otimes b
\end{align*}
is compact. 
This follows from \cite[Theorem 4.1]{simon79} after a few steps. First we let $u_m\in C_c^\infty(\R)$ be an approximate unit for $C_0(\R)$. Then $u_mf$ and $u_ma$ are $L^2$-functions and so
\[
(u_mf)(D_+) (u_ma)(\ln{(L)}) \otimes b
\]
is compact by \cite[Theorem 4.1]{simon79}, and since $D_+$ and $\frac12 \ln{(L)}$ are canonically conjugate. Now because $f$ and $a$ vanish at $\pm \infty$ we readily see that $(u_mf)(D_+)\to f(D_+)$ in operator norm, and similarly for $(u_ma)(\ln{(L)})$. As the compacts are norm closed, the limit is compact. The final claim follows by writing $h = g \circ ( \ln)$.
\end{proof}
\begin{lemma}
Let $n = 2$ and $V$ satisfy Assumption \eqref{ass11} for some $\rho > 11$ and suppose that there are no $p$-resonances. Then $W_-$ is of the form given in Equation \eqref{eq:waveopform}.
\end{lemma}
\begin{proof}
Consider the difference
\begin{align*}
X &= F_0 \left(\frac12\left(\textup{Id}+\tanh{\left(\frac{\pi D_2}{2}\right)}\right) \eta(H_0) - \frac12\left(\textup{Id}+\tanh{(\pi D_2)}-i\cosh{(\pi D_2)}^{-1}\right) \eta(H_0)\right)F_0^* \\
%&= \frac12\left(\textup{Id}+\tanh{\left(\pi D_+\right)}\right) \eta(L) - \frac12\left(\textup{Id}+\tanh{(2D_+)}-i\cosh{(2D_+)}^{-1}\right) \eta(L) \\
&= \frac12\left(\textup{Id}+\tanh{\left(\pi D_+\right)}-\textup{Id}-\tanh{(2\pi D_+)}+i\cosh{(2\pi D_+)}^{-1}\right) \eta(L) \\
&=: Y(D_+) \eta(L).
\end{align*}
The function $\eta$ vanishes at $0$ and $ \infty$ and the function $Y$ is continuous and square integrable. Thus an application of Lemma \ref{lem:f(x)g(nabla)} gives that the operator $X$ is compact. Hence $W_-$ has the form given in Equation \eqref{eq:waveopform}.
\end{proof}

\subsection{Setting up the proof}

We will now provide a number of preparatory results before proceeding to the proof of Theorem \ref{thm:main} for $n\geq 4$: the case $n=3$ is proved in \cite{richard13}. 

The proof is broken up into a number of steps, whose ultimate goal is to factorise (in the spectral representation) the wave operator into the composition of several operators, whose mapping properties as operators between various weighted Sobolev spaces must be analysed. The properties of these factor operators depend heavily on their low energy behaviour and we use the resolvent expansions of Section \ref{sec:resolvents} to determine this behaviour.

\begin{defn}
\label{defn:phi-eps}
Define, for $\eps > 0$ and $\lambda \in \R$, the operator $\vp_\eps(H_0-\lambda) = \frac{\eps}{\pi} R_0(\lambda \mp i \eps) R_0(\lambda \pm i\eps)$ on $L^2(\R^n)$. 
Similarly for $M=F_0H_0F_0^*$ the operator of multiplication by the spectral variable, set $\vp_\eps(M-\lambda) = \frac{\eps}{\pi} (M-(\lambda \mp i \eps))^{-1} (M-(\lambda \pm i\eps))^{-1}$.
\end{defn}

By \cite[Section 1.4]{yafaev92}, the limits $\lim_{\eps \to 0}{\langle \vp_\eps(H_0-\lambda) f, g \rangle}$ exist for almost every $\lambda \in \R$ and $f,g\in L^2(\R^n)$. Moreover the limit satisfies
\begin{align*}
\langle f, g \rangle &= \int_\R{\left(\lim_{\eps \to 0}{\big\langle \vp_\eps(H_0-\lambda) f, g \big\rangle } \right)\, \d \lambda}.
\end{align*}

Using the resolvent identity $R(z) = (\textup{Id} +R_0(z)V)^{-1}R_0(z)$ and the stationary formula for the wave operator in Equation \eqref{eq:statwaveop} we can then show that
\begin{align*}
\langle (W_\pm - \textup{Id} ) f, g \rangle  &= - \int_\R{\left( \lim_{\eps \to 0}{\big\langle \vp_\eps(H_0 - \lambda) f, (\textup{Id} +VR_0(\lambda \pm i \eps))^{-1} V R_0(\lambda \pm i \eps) g \big\rangle } \right)\, \d \lambda}.
\end{align*}

We now aim to derive formulas for the wave operators in the spectral representation of $H_0$, that is, $F_0(W_--\textup{Id} )F_0^*$. Recall from Lemma \ref{lem:eff-emm}  the operator $M=F_0 H_0 F_0^*$ of multiplication by the spectral variable. For suitable $f, g \in \H_{spec}$ we compute that
\begin{align*}
&-\langle F_0(W_\pm - \textup{Id})F_0^* f, g \rangle_{\H_{spec}} \\
&= \!\int_\R{\!\left(\lim_{\eps \to 0}{\big\langle V(\textup{Id} +R_0(\lambda \mp i \eps) V)^{-1} F_0^* \vp_\eps(M-\lambda) f, F_0^* (M-\lambda \mp i \eps)^{-1} g \big\rangle_{\H}} \right)\, \d \lambda} \\
&= \!\int_\R{\!\left(\lim_{\eps \to 0}{\int_0^\infty{\!\!\!\!\Big\langle\!\! \left[F_0 V(\textup{Id} +R_0(\lambda \mp i \eps) V)^{-1} F_0^* \vp_\eps(M-\lambda) f \right]\!(\mu,\cdot), (\mu - \lambda \mp i \eps)^{-1} g(\mu,\cdot) \Big\rangle_{\mathcal{P}}\, \d \mu}} \right) \!\d \lambda}.
\end{align*}

With $T(z) = V(\textup{Id} +R_0(z)V)^{-1} = V - VR(z) V$ for $z \in \C \backslash \R$ we have
\begin{align}
\label{eq:long}
&\langle F_0^*(W_\pm - \textup{Id})F_0^* f, g \rangle_{\H_{spec}} \nonumber \\
&= -\int_\R{\left(\lim_{\eps \to 0}{\int_0^\infty\!\!{\big\langle\left[F_0 T(\lambda \mp i \eps) F_0^* \vp_\eps(M-\lambda) f \right](\mu,\cdot), (\mu - \lambda \mp i \eps)^{-1} g(\mu,\cdot) \big\rangle_{\mathcal{P}}\, \d \mu}} \right)\,\d \lambda}.
\end{align}

To interchange the limit as $\eps \to 0$ and the integral over $\mu$ of Equation \eqref{eq:long}, which we do in 
subsection \ref{sub:interchange}, we first need to analyse $F_0T(\lambda+i0)F_0^*$. 

\subsection{Analysing the diagonalisation maps}
Recall that for $\lambda > 0$ the operator $\Gamma_0(\lambda)$ (see subsection \ref{subsec:ass-notes}) extends to an element
of $\mathcal{B}(H^{s,t}, \mathcal{P})$ for each  $s \in \R$ and $t > \frac12$, as well as $\lambda \mapsto \Gamma_0(\lambda) \in \mathcal{B}(H^{s,t}, \mathcal{P})$ is continuous.  

We need stronger continuity and boundedness results for $\Gamma_0$.

\begin{lemma}\label{lem:bounded}
Let $n \geq 4$,  $s \geq 0$ and $t > \frac{n}{2}$. Then, the functions $G_\pm: (0,\infty) \to \mathcal{B}(H^{s,t}, \mathcal{P})$ defined by $G_+(\lambda) = \lambda^{ \frac{1}{4}} \Gamma_0(\lambda)$ and $G_-(\lambda) = \lambda^{-\frac{1}{4}} \Gamma_0(\lambda)$ are continuous and bounded.
\end{lemma}
\begin{proof}
The continuity is immediate from the previous paragraph. For the boundedness, it is sufficient to check the boundedness of $G_-$ in a neighbourhood of $0$ and $G_+$ in a neighbourhood of $\infty$. For the first bound, we need the asymptotic development of $\gamma(\lambda^\frac12)$ as $\lambda \to 0$. We recall that for $f \in C_c^\infty(\R^n)$ we have
\begin{align*}
[\gamma(\lambda^\frac12) \Fi f](\omega) &= \sum_{j=0}^K{(i\lambda^\frac12)^j [\gamma_j f](\omega)}+O\left(\lambda^{\frac{K+1}{2}} \right).
\end{align*}

By Lemma \ref{lem:gammabounded} we have for $t > j+\frac{n}{2}$ that $\gamma_j \in \mathcal{B}(H^{0,t}, \mathcal{P})$. From this expansion we see that for $t > j+\frac{n}{2}$ and for any $\alpha \leq \frac{n-2}{4}$ we have
\begin{align}\label{eq:gammalim}
\lambda^{-\alpha} \norm{\Gamma_0(\lambda)} &= \lambda^{\frac{n-2}{4}-\alpha} 2^{-\frac12} \norm{\gamma(\lambda^\frac12)}
\end{align}
is bounded in a neighbourhood of zero. Taking $\alpha = \frac{1}{4}$ gives immediately the statement for $G_-$, since $n \geq 3$.  For the second statement, we refer to \cite[Theorem 1.1.4]{yafaev10} for the estimate, valid for $t > \frac12$, 
\begin{align}\label{est:yaf}
r^{n-1} \int_{\Sf^{n-1}}{|u(r\omega)|^2\, \d \omega} &\leq C \norm{u}^2_{H^{t,0}},
\end{align}
from which we obtain (with $r = \lambda^\frac12$ and $u = \Fi f$) that
\begin{align*}
\norm{\Gamma_0(\lambda)f}_{\mathcal{P}}^2 &= \int_{\Sf^{n-1}}{2^{-1} \lambda^{\frac{n-2}{2}} |[\Fi f](\lambda^\frac12 \omega)|^2\, \d \omega} 
= 2^{-1}\lambda^{-\frac{1}{2}} \int_{\Sf^{n-1}}{\lambda^{\frac{n-1}{2}} |[\Fi f](\lambda^\frac12 \omega)|^2\, \d \omega} \\
&\leq 2^{-1} \lambda^{-\frac12} C \norm{\Fi f}^2_{H^{t,0}} 
= 2^{-1} \lambda^{-\frac12} C \norm{f}^2_{H^{0,t}}.
\end{align*}
This gives us the estimate
\begin{align*}
\lambda^\frac{1}{4} \norm{\Gamma_0(\lambda)f} &\leq \tilde{C} \norm{f}_{H^{0,t}},
\end{align*}
so that $G_+$ is bounded as $\lambda \to \infty$. For $s \geq 0$, we use the inclusion $H^{s,t} \subset H^{0,t}$.
\end{proof}

From the arguments above we immediately obtain $\lambda \mapsto \norm{\Gamma_0(\lambda)}_{\mathcal{B}(H^{s,t},\mathcal{P})}$ is continuous and bounded for $s \geq 0$ and $t > \frac12$. We can further strengthen Lemma \ref{lem:bounded} for $G_-$.

\begin{lemma}\label{lem:limits}
Let $s \geq 0$ and $t > \frac{n}{2}$. Then $\Gamma_0(\lambda) \in \mathcal{K}(H^{s,t}, \mathcal{P})$ for any $\lambda \in \R^+$ and the function $G_-:\R^+ \to \mathcal{K}(H^{s,t}, \mathcal{P})$ is continuous and vanishes as $\lambda \to \infty$ and as $\lambda \to 0$.
\end{lemma}
\begin{proof}
The compactness statement follows from the compact embedding $H^{s,t} \subset H^{q,r}$ for any $q < s$, $r < t$. See \cite[Proposition 4.1.5]{amrein96}.

For the continuity, the same argument as the previous lemma works. That $G_-$ vanishes as $\lambda \to 0$ follows from Equation \eqref{eq:gammalim} and the fact that $\gamma(\lambda^\frac12)\F \to \gamma_0 \F$ in the norm of $\mathcal{B}(H^{s,t},\mathcal{P})$ as $\lambda \to 0$. It remains to check the limit as $\lambda \to \infty$. The estimate \eqref{est:yaf} of the previous lemma tells us the behaviour of $\Gamma_0(\lambda)$ as $\lambda \to \infty$ in the norm of $\mathcal{B}(H^{0,t})$, so we analyse an equivalent limit in this norm. Define the operator $P_s = (\textup{Id}+H_0)^{-\frac{s}{2}}$ and note that for $\alpha \geq 0$ we have
\begin{align*}
\lambda^{-\alpha} \Gamma_0(\lambda) P_s &= \lambda^{-\alpha} (1+\lambda)^{-\frac{s}{2}} \Gamma_0(\lambda)
\end{align*}
for each $\lambda \in \R^+$. Thus we find 
\begin{align*}
\lim_{\lambda \to \infty}{\norm{\lambda^{-\alpha} \Gamma_0(\lambda)}_{\mathcal{B}(H^{s,t})}} &= 0
\end{align*}
if and only if 
\begin{align*}
\lim_{\lambda \to \infty}{\norm{\lambda^{-\alpha}(1+\lambda)^{-\frac{s}{2}} \Gamma_0(\lambda)}_{\mathcal{B}(H^{0,t})}} &= 0.
\end{align*}
Hence the claim follows from the estimate \eqref{est:yaf} of the previous lemma as long as $-\alpha-\frac{s}{2} < \frac{1}{4}$, which is the same as $s > - \frac12-2\alpha$. Letting $\alpha = \frac{1}{4}$ gives the desired result.
\end{proof}
We now consider the multiplication operator $G_-: C_c(\R^+,H^{s,t}) \to \H_{spec}$ given by
\begin{align}\label{eq:G-minus}
[G_- f](\lambda,\omega) &:= [G_-(\lambda) f](\lambda,\omega) = \lambda^{-\frac{1}{4}} [\Gamma_0(\lambda) f](\lambda,\omega).
\end{align}
Lemmas \ref{lem:bounded} and \ref{lem:limits} show that the operator $G_-$ extends, for $s \geq 0$ and $t > \frac12$ to an element of $\mathcal{B}(L^2(\R^+, H^{s,t}), \H_{spec})$.

The next step is to deal with the limit as $\eps \to 0$ of $\vp_\eps(M-\lambda)$ from Definition \ref{defn:phi-eps}. %For this purpose we need the continuous extension of the inner product $\langle \cdot, \cdot \rangle $ to a duality $\langle \cdot, \cdot \rangle_{H^{s,t},H^{-s,-t}}$.

\begin{lemma}\label{lem:converge}
Fix $n \geq 4$. Take $s \geq 0$ and $t > \frac{n}{2}$. For $\lambda \in \R^+$ and $f \in C_c(\R^+, \mathcal{P})$ we have
\begin{align*}
\lim_{\eps \to 0}{\norm{F_0^* \vp_\eps(M-\lambda) f - \Gamma_0(\lambda)^*f(\lambda)}_{H^{-s,-t}}} &= 0.
\end{align*}
\end{lemma}
We omit the proof, since it is identical to that of \cite[Lemma 2.3]{richard13}.

From now on we only consider the operator $W_-$ since our resolvent results have only been in the upper half plane. Similar results could be derived in the lower half plane, however the operator $W_+$ can be obtained in a much simpler manner from the relation $W_+ = W_-S^*$ at a later stage. 

\begin{lemma}
\label{lem:mult}
Fix $n \geq 4$. For $z \notin \sigma(H_0)$ let $T(z) = V(\textup{Id}+R_0(z)V)^{-1}$. Suppose that  $\rho, t$ and $V$ satisfy Assumption \ref{ass:best-ass}, and that there are no resonances for $n = 4$.
%If $n = 3$, take $\rho > 5$ and $t \in (\frac{5}{2}, \rho - \frac{5}{2})$. If $n = 5$, take $\rho > 7$ and $\frac{7}{2} < t < \rho - \frac{7}{2}$. If $n = 7,8,9,10$, take $\rho > 6$ and $6 - \frac{n}{2} < t < \rho - (6-\frac{n}{2})$. If $n \geq 11$ is odd, take $\rho > \frac{n+1}{2}$ and $\frac12 < t < \rho - \frac12$. If $n = 6$, let $\rho > 6$ and $3 < t < \rho - 3$. If $n \geq 12$ is even, let $\rho > \frac{n+1}{2}$ and $\frac12 < t < \rho - \frac12$. 
Then the function $B:\R^+ \to \mathcal{B}(\mathcal{P},H^{0, \rho-t})$ defined by 
\begin{align*}
B(\lambda) &= \lambda^{\frac{1}{4}} T(\lambda + i0) \Gamma_0(\lambda)^*
\end{align*}
is continuous and bounded, and the multiplication operator $B:C_c(\R^+,\mathcal{P}) \to L^2(\R^+, H^{0,\rho-t})$ given by
\begin{align*}
[B f](\lambda) &= B(\lambda)f(\lambda)
\end{align*}
extends to an element of $\mathcal{B}(\H_{spec}, L^2(\R^+,H^{0,\rho-t}))$.
\end{lemma}
\begin{proof}
The continuity of $B$ follows from the limiting absorption principle \cite[Theorem 5.2.7]{kuroda78}. For boundedness, it will be sufficient to show that the map $\lambda \mapsto \norm{B(\lambda)}_{\mathcal{B}(\mathcal{P},H^{0,\rho-t})}$ is bounded in a neighbourhood of $0$ and in a neighbourhood of $\infty$.

For $\lambda \geq 1$, Lemma \ref{lem:resolvents} shows that the function $\lambda \mapsto \norm{T(\lambda + i0)}_{\mathcal{B}(H^{0,-t}, H^{0, \rho - t})}$ is bounded provided $\rho > \frac{n+1}{2}$, which is guaranteed by our assumptions. We also know from Lemma \ref{lem:bounded} that the function $\R^+ \ni \lambda \mapsto \lambda^\frac{1}{4} \norm{\Gamma_0(\lambda)^*}_{\mathcal{B}(\mathcal{P}, H^{0,-t})}$ is bounded.  This shows that $\lambda \mapsto \norm{B(\lambda)}_{\mathcal{B}(\mathcal{P},H^{0,\rho-t})}$ is bounded in a neighbourhood of $\infty$.

For $\lambda$ in a neighbourhood of $0$, we use asymptotic developments for $T(\lambda+i0)$ and $\Gamma_0(\lambda)^*$. The asymptotics of $\Gamma_0(\lambda)$ are discussed in Lemma \ref{lem:gammaexp}. By taking adjoints in our Lemma \ref{lem:gammaexp}, we see that for each $s \geq 0$ we have the expansion
\begin{align*}
\Gamma_0(\lambda)^* &= 2^{-\frac12} \lambda^{\frac{n-2}{4}} (\gamma_0^* -i \lambda^\frac12 \gamma_1^* + O(\lambda)^\frac12)
\end{align*}
in $\mathcal{B}(\mathcal{P}, H^{s,-t})$ as $\lambda \to 0$.

The form of the asymptotic development of $T(\lambda + i0)$ as $\lambda \to 0$ depends heavily on the (non)existence of zero energy eigenvalues and zero energy resonances: the specific statements we use are summarised in Section \ref{sec:resolvents}. 

%{\bf
%In the case $n = 3$, the most singular behaviour \cite[Lemma 4.5]{jensen79} is given by 
%%the following \cite[Lemma 4.5]{jensen79}. The operator $T(\lambda + i0)$ has the expansion
%\begin{align*}
%T(\lambda + i 0) &= \lambda^{-1} VP_0 V - i \lambda^{-\frac12} C +O(1)
%\end{align*}
%in $\mathcal{B}(H^{1,-t}(\R^3), H^{0,t}(\R^3))$ as $\lambda \to 0$, with $P_0$ the orthogonal projection onto $\ker(H)$ and $C \in \mathcal{B}(H^{1,-t}(\R^3), H^{0,t}(\R^3))$. Using these expressions for $\Gamma_0(\lambda)^*$ and $T(\lambda+i0)$, one can write $\lambda^{\frac{1}{4}} T(\lambda+i0) \Gamma_0(\lambda)^*$ as a sum of terms bounded in $\mathcal{B}(\mathcal{P}, H^{0,\rho - t})$ as $\lambda \to 0$ and a term $2^{-\frac12} \lambda^{-\frac{1}{2}} VP_0V \gamma_0^*$, which appears unbounded. However the proof of \cite[Theorem 5.4]{jensen79} shows that $P_0 V\gamma_0^* = 0$. Thus all the terms in the asymptotic development of $\lambda^\frac{1}{4} T(\lambda+i0) \Gamma_0(\lambda)^*$ are bounded in $\mathcal{B}(\mathcal{P}, H^{0,\rho - t)}(\R^3))$ TYPO??. The statements regarding $B$ thus follow immediately in the case $n = 3$.}

For $n \geq 5$ the behaviour as $\lambda \to 0$ of $T(\lambda+i0)$ is described explicitly in Lemma \ref{lem:resexpbig} as
\begin{align*}
T(\lambda+i0) &= \lambda^{-1} VP_0 V + o(1)
\end{align*}
in $\mathcal{B}(H^{1,-t},H^{0,-t})$ as $\lambda \to 0$. The lowest order power of $\lambda$ in the expansion of $\lambda^\frac{1}{4}T(\lambda+i0)\Gamma_0(\lambda)^*$ is thus the term $-\lambda^{\frac{n-5}{4}}\gamma_0 V P_0 V \gamma_0^*$. So we see that $\lambda^\frac{1}{4}T(\lambda+i0)\Gamma_0(\lambda)^*$ is bounded in $\mathcal{B}(\mathcal{P},H^{0,\rho-t})$ as $\lambda \to 0$ for $n \geq 5$.
%Unlike in dimension $n \leq 4$ we do not have the relation $P_0 V 1 = 0$ \cite[Remark 8.3]{jensen80}, however we still have enough decay to control the function $B$ near zero. 

We now consider $n = 4$. For the asymptotic development of $T(\lambda+i0)$, we appeal to Lemma \ref{lem:res4d}. We obtain (with $\psi$ the normalised zero energy resonance of Definition \ref{defn:res4d} and $a$ the constant defined by Equation \eqref{eq:constantdefn}) the expansion
\begin{align*}
T(\lambda+i0) &= \lambda^{-1} VP_0 V +\lambda^{-1}\left(a-\ln(\lambda)-i\frac{\pi}{2}\right)^{-1} \langle V\psi,\cdot \rangle V\psi - \ln(\lambda) C_1 +O(1)
\end{align*}
as $\lambda \to 0$ in $\mathcal{B}(H^{1,-t}, H^{0,t})$. Hence we find as $\lambda \to 0$ that
\begin{align*}
&\lambda^\frac{1}{4}T(\lambda+i0) \Gamma_0(\lambda)^* \\
&= 2^{-\frac12} \lambda^{\frac{3}{4}} \left(\lambda^{-1} VP_0 V +\lambda^{-1}\left(a-\ln(\lambda)-i\frac{\pi}{2}\right)^{-1} \langle  V\psi,\cdot \rangle V\psi - \ln(\lambda) C_1 +O(1)\right) \\
&\times \left(\gamma_0^*-i\lambda^\frac12 \gamma_1^*+o(\lambda^\frac12)\right) \\
&= 2^{-\frac12} \left(\lambda^{-\frac{1}{4}} VP_0 V \gamma_0^*+\lambda^{-\frac{1}{4}}\left(a-\ln(\lambda)-i\frac{\pi}{2}\right)^{-1}\langle V\psi,\cdot \rangle V\psi \gamma_0^* + o(\lambda^\frac{1}{4} \ln(\lambda))\right).
\end{align*}
There are two terms of concern here as $\lambda \to 0$. The first is the term $\lambda^{-\frac{1}{4}} VP_0 V \gamma_0^*$, which vanishes by \cite[Lemma 3.3]{jensen84}. The second is the term containing $\psi$, which blows up as $\lambda \to 0$. Thus we must make the assumption $\psi = 0$, that is there are no resonances.
\end{proof}

The case $n = 2$ has been resolved in \cite{richard13ii} in the non-resonant case and \cite{richard21} in the case of $s$-resonances but not $p$-resonances. The presence of certain resonances in dimensions $n = 2$ and $n = 4$ is well-known to produce complicated behaviour in scattering theory, see for instance \cite[Theorem III.3]{guillope81} and \cite[Theorem 6.3]{bolle85}, and so it is not surprising that we must exclude them in our analysis for dimension $n=4$, as is done in \cite{richard21} for dimension $n=2$.

We also note the following essential corollary of Lemmas \ref{lem:limits} and \ref{lem:mult}.

\begin{cor}
Suppose that $V$, $\rho$ and $t$ satisfy Assumption \ref{ass:best-ass}. Then for $n \geq 2$ we have the equality $S(\infty) = \textup{Id}$ in $\mathcal{B}(\mathcal{H})$.
\end{cor}
\begin{proof}
For dimension $n \geq 4$, we note that the operators $G_-(\lambda)$ and $B(\lambda)$ satisfy $G_-(\lambda)B = S(\lambda)-\textup{Id}$ and thus, since $G_-$ vanishes as $\lambda \to \infty$ and $B$ is bounded we have the desired result. In dimension $n = 3$, a combination of \cite[Lemmas 2.2 and 2.4]{richard13} in an analogous fashion shows that $S(\infty) = \textup{Id}$. In dimension $n=2$, the statement is an immediate corollary of \cite[Theorem 13]{richard21}.
\end{proof}

\subsection{Interchanging limits}
\label{sub:interchange}

In this subsection we interchange the integral over $\lambda$ and the limit as $\eps \to 0$ in Equation \eqref{eq:long}.

For fixed $\eps, \lambda > 0$ and $s \in \R$ we define for $f,g  \in C_c^\infty(\R^+,\mathcal{P})$ the following functions
\begin{align}\label{eq:func defn}
h_\eps(\lambda) &:= \lambda^\frac{1}{4} T(\lambda+i\eps)F_0^* \vp_\eps(M-\lambda) f \in H^{s,0},\quad \textup{ and } \\
p(\lambda) &:= \lambda^{-\frac{1}{4}} \Gamma_0(\lambda)^* g(\lambda) \in H^{-s,0}. \nonumber
\end{align}

We also denote by 
\begin{equation}
q_\lambda(\nu) = \begin{cases}\left(\frac{\nu+\lambda}{\lambda} \right)^\frac{1}{4} p(\nu+\lambda) &\nu>-\lambda\\
0&\nu\leq -\lambda\end{cases}
\label{eq:kew}
\end{equation}
the extension by zero of the function $(-\lambda,\infty) \ni \nu \mapsto \left(\frac{\nu+\lambda}{\lambda} \right)^\frac{1}{4} p(\nu+\lambda) \in H^{0,-s}$ to all of $\R$.

\begin{lemma}
\label{lem:part1}
Let $n \geq 4$ and suppose that $\rho, t$ satisfy Assumption \ref{ass:best-ass}, and that there are no resonances for $n = 4$. If $|V(x)| \leq C(1+|x|)^{-\rho}$ then one has the equality
\begin{align*}
\langle F_0 (W_--\textup{Id} ) F_0^* f, g \rangle_{\H_{spec}} &= -i \int_{\R^+}{ \left\langle h_0(\lambda), \int_0^\infty{\e^{i\nu z} q_\lambda(\nu)\, \d \nu}\right\rangle_{H^{0,s},H^{0,-s}}\, \d \lambda}.
\end{align*}
\end{lemma}
\begin{proof}
Take $f \in C_c(\R^+,\mathcal{P})$ and $g \in C_c^\infty(\R^+) \odot C(\Sf^{n-1})$, and set $m = \rho - t$. 
We may then write the expression \eqref{eq:long}, using the functions $h_\eps$ and $p$ defined in Equation \eqref{eq:func defn}, as
\begin{align*}
&\langle F_0(W_--\textup{Id} )F_0^* f, g \rangle \\
&= - \int_\R{\left(\lim_{\eps \to 0}{\int_0^\infty{\langle T(\lambda + i\eps) F_0^*\vp_\eps(M-\lambda) f, (\mu-\lambda+i\eps)^{-1} \Gamma_0(\mu)^*g(\mu) \rangle_{H^{0,m},H^{0,-m}}\, \d \mu}} \right)\, \d \lambda} \\
&= - \int_{\R^+}{\left(\lim_{\eps \to 0}{\int_{\R^+}{\left\langle h_\eps(\lambda), \frac{\lambda^{-\frac{1}{4}} \mu^{\frac{1}{4}}}{\mu-\lambda+i\eps} p(\mu) \right\rangle_{H^{0,m},H^{0,-m}}\, \d \mu}} \right)\, \d \lambda}.
\end{align*}
We now use the formula
\begin{align*}
(\mu-\lambda+i\eps)^{-1} &= -i \int_0^\infty{\e^{i(\mu-\lambda)z}\e^{-\eps z}\, \d z}
\end{align*}
and Fubini's theorem to obtain 
\begin{align}\label{eq:domconv}
&\lim_{\eps \to 0}{\int_0^\infty{\left\langle h_\eps(\lambda), \frac{\lambda^{-\frac{1}{4}} \mu^{\frac{1}{4}}}{\mu-\lambda+i\eps} p(\mu) \right\rangle_{H^{0,m},H^{0,-m}}\, \d \mu}}\nonumber \\
&= -i \lim_{\eps \to 0}{\int_0^\infty{\e^{-\eps z} \left(\left\langle h_\eps(\lambda), \int_0^\infty{\e^{i(\mu-\lambda)z} \lambda^{-\frac{1}{4}} \mu^\frac{1}{4} p(\mu)  \, \d \mu} \right\rangle_{H^{0,m},H^{0,-m}} \right) \, \d z}} \nonumber \\
&= -i \lim_{\eps \to 0}{\int_0^\infty{\e^{-\eps z} \left(\left\langle h_\eps(\lambda), \int_{-\lambda}^\infty{\e^{i\nu z} \left(\frac{\nu+\lambda}{\lambda}\right)^\frac{1}{4} p(\nu+\lambda)  \, \d \nu} \right\rangle_{H^{0,m},H^{0,-m}} \right) \, \d z}},
\end{align}
where in the last line we have made a change of variables to recognise a Fourier transform. The $z$-integrand of \eqref{eq:domconv} can be bounded independently of $\eps \in (0,1)$ explicitly with
\begin{align}\label{ineq:estimate}
&\left|\e^{-\eps z} \left\langle h_\eps(\lambda), \int_{-\lambda}^\infty{\e^{i\nu z} \left(\frac{\nu+\lambda}{\lambda}\right)^\frac{1}{4} p(\nu+\lambda)  \, \d \nu} \right\rangle_{H^{0,m},H^{0,-m}} \right| \nonumber \\
&\leq \norm{h_\eps(\lambda)}_{H^{0,m}} \norm{\int_{-\lambda}^\infty{\e^{i\nu z} \left( \frac{\nu+\lambda}{\lambda} \right)^{\frac{1}{4}} p(\nu+\lambda)\, \d \nu}}_{H^{0,-m}} \nonumber \\
&:= \norm{h_\eps(\lambda)}_{H^{0,m}} j_\lambda(z),
\end{align}
where we have used $\e^{-\eps z} \leq 1$.
We also know from Lemma \ref{lem:converge} that as $\eps \to 0$, $h_\eps(\lambda)$ converges to $\lambda^\frac{1}{4} T(\lambda+i0) \Gamma_0(\lambda)^* f(\lambda)$ in $H^{0,m}$. Therefore the family $\norm{h_\eps(\lambda)}_{H^{0,m}}$ (and thus the entire expression \eqref{ineq:estimate}) is bounded by a constant independent of $\eps \in (0,1)$.

In order to exchange the integral over $z$ and the limit as $\eps \to 0$ in Equation \eqref{eq:domconv}, it remains to show that the function $j_\lambda$ of \eqref{ineq:estimate} is in $L^1(\R^+, \d z)$. Recall  $q_\lambda$ from Equation \eqref{eq:kew}. Note that $j_\lambda$ in \eqref{ineq:estimate} can be rewritten as $j_\lambda(z) = (2\pi)^\frac12 \norm{[\Fi^* q_\lambda](z)}_{H^{0,-m}}$ (here the Fourier transform is one-dimensional).

To estimate $j_\lambda$, we denote by $D = -i \pd{}{\nu}$ the self-adjoint operator acting (densely) on $L^2(\R)$ and by $P = (\textup{Id}+D^2)$. Then we have
\begin{align}\label{eq:fourier-P}
\norm{[\Fi^* q_\lambda](z)}_{H^{0,-m}} &= (1+z^2)^{-1} \norm{[\Fi^* P q_\lambda](z)}_{H^{0,-m}}.
\end{align}
The function $z \mapsto (1+z^2)^{-1}$ is in $L^1(\R^+, \d z)$ and thus it suffices to show $\norm{[\Fi^* P q_\lambda](z)}_{H^{0,-m}}$ is bounded independently of $z$. 
Suppose that $g = \eta \otimes \xi \in C_c^\infty(\R^+) \otimes C(\mathcal{P})$ is a simple tensor, so that we have
\begin{align*}
\left(\frac{\nu+\lambda}{\lambda}\right)^\frac{1}{4}[p(\nu+\lambda)](x) &=  2^{-\frac12} (2\pi)^{-\frac{n}{2}} \lambda^{-\frac{1}{4}} (\nu+\lambda)^\frac{n-2}{4} \eta(\nu+\lambda) \int_{\Sf^{n-1}}{\e^{i(\nu+\lambda)^\frac12 \langle x, \omega \rangle}\xi(\omega) \, \d \omega}.
\end{align*}
Differentiating with respect to $\nu$ twice shows that there are $\eta_{j,\lambda} \in C_c^\infty(\R^+)$ such that
\begin{align*}
[Pq_\lambda](\nu)(x) &=\eta_{1,\lambda}(\nu) \int_{\Sf^{n-1}}{\e^{i(\nu+\lambda)^\frac12 \langle x, \omega \rangle}\xi(\omega) \, \d \omega}  + \eta_{2,\lambda}(\nu) \int_{\Sf^{n-1}}{\langle x, \omega \rangle \e^{i(\nu+\lambda)^\frac12 \langle x, \omega \rangle}\xi(\omega) \, \d \omega}\\
& +\eta_{3,\lambda}(\nu) \int_{\Sf^{n-1}}{(\langle x, \omega \rangle)^2 \e^{i(\nu+\lambda)^\frac12 \langle x, \omega \rangle}\xi(\omega) \, \d \omega}.
\end{align*}
Then estimating each term individually we obtain $C_1,C_2,C_3,C>0$ such that
\begin{align*}
|[\F^* P q_\lambda](z)(x)| &\leq C_1 +C_2 |x| + C_3 |x|^2 
%&\leq C'(1+|x|+|x|^2) \\
%&\leq \begin{cases} C'(2+|x|^2), \quad & \textup{ if } |x| \leq 1, \\
%C'(1+2|x|^2), \quad & \textup{ if } |x| \geq 1, \end{cases} \\
\leq C(1+|x|^2)
\end{align*}
which gives
\begin{align*}
|[[\Fi^* P q_\lambda](z)](x)| &\leq C (1+|x|^2).
\end{align*}
For $\tau(x) = (1+|x|^2)$, we compute that
\begin{align*}
\norm{\tau}_{H^{0,-m}}^2 &= \int_{\R^n}{(1+|x|^2)^{-\frac{m}{2}}(1+|x|^2)\, \d x} =\textup{Vol}(\Sf^{n-1}) \int_0^\infty{r^{n-1} (1+r^2)^{1 -\frac{m}{2}}\, \d r},
\end{align*}
which is finite for $n -m-1 < -1$, or $m> n+2$. Thus $\tau \in H^{0,-m}$ for $m >n+2$.

We note in passing that the condition on $m = \rho - t$ requires $\rho > n+t+ 2$, which is guaranteed by Assumption \ref{ass:best-ass}. 

By linearity, for each $g \in C_c^\infty(\R^+) \odot C(\Sf^{n-1})$ we find that $\norm{[\Fi^* P q_\lambda](z)}_{H^{0,-m}}$ is bounded independently of $z$. Thus using Equation \eqref{eq:fourier-P} we have $\norm{[\Fi^* q_\lambda](z)}_{H^{0,-m}}\leq \tilde{C}(1+z^2)^{-1}$ and so  $\norm{[\Fi^* q_\lambda](z)}_{H^{0,-m}}\in L^1(\R^+)$.

As a consequence, we can apply Lebesgue's dominated convergence theorem to exchange the limit as $\eps \to 0$ with the integral over $z$ in Equation \eqref{eq:domconv} and obtain
\begin{align*}
& -i \lim_{\eps \to 0}{\int_0^\infty{\e^{-\eps z} \left(\left\langle h_\eps(\lambda), \int_{-\lambda}^\infty{\e^{i\nu z} \left(\frac{\nu+\lambda}{\lambda}\right)^\frac{1}{4} p(\nu+\lambda)  \, \d \nu} \right\rangle_{H^{0,m},H^{0,-m}} \right) \, \d z}} \\
&= -i \left\langle h_0(\lambda), \int_0^\infty{\int_\R{\e^{i\nu z} q_\lambda(\nu)\, \d \nu}\, \d z}\right\rangle_{H^{0,m},H^{0,-m}}.
\end{align*}
We  have thus shown (on a dense set of $f, g$) the equality
\begin{align*}
\langle F_0 (W_--\textup{Id} ) F_0^* f, g \rangle_{\H_{spec}} &= -i \int_{\R^+}{ \left\langle h_0(\lambda), \int_0^\infty{\e^{i\nu z} q_\lambda(\nu)\, \d \nu}\right\rangle_{H^{0,m},H^{0,-m}}\, \d \lambda}. \qedhere
\end{align*}
\end{proof}

\begin{rmk}
The proof of Lemma \ref{lem:part1} shows how the various requirements on $\rho$ in Assumption \ref{ass:best-ass} are obtained. In particular, in Lemma \ref{lem:resolvents} we assumed $\rho > \frac{n+1}{2}$ and in Lemma \ref{lem:gammabounded} we required $t > \frac{n}{2}$. In Lemma \ref{lem:part1} we required $\rho > n+t+2$ which gives $\rho > \frac{3n+4}{2}$ for $n \geq 5$. The case $n = 4$ requires $\rho > 12$ to use Lemma \ref{lem:res4d}. 
\end{rmk}

\subsection{Identification of function of dilation, completion of proof}

Recall now the operator $D_+$, the self-adjoint generator of dilations on $L^2(\R^+)$. We can easily take functions of $D_+$ using the group structure via the formula
\begin{align*}
[\psi(D_+)f](x) &= (2\pi)^{-\frac12} \int_\R{ [\Fi^* \psi](t) [U_+(t) f](x)\, \d t}.
\end{align*}
We introduce the function $\psi \in C(\R) \cap L^\infty(\R)$ defined by
\begin{align}\label{eq:dilationfunc}
\psi(x) &= \frac12 \left(1-\tanh{(2\pi x)}-i \cosh{(2\pi x)}^{-1} \right).
\end{align}
The Hilbert spaces $L^2(\R^+,H^{s,t})$ and $\H_{spec}$ can be naturally identified with the Hilbert spaces $L^2(\R^+) \otimes H^{s,t}$ and $L^2(\R^+) \otimes \mathcal{P}$. We recall the operator $G_-$ of Equation \eqref{eq:G-minus} and $B$ from Lemma \ref{lem:mult}.

\begin{thm}
\label{thm:hardwork}
Let $n \geq 4$ and suppose that $\rho, t$ satisfy Assumption \ref{ass:best-ass}, and that there are no resonances for $n = 4$. If $|V(x)| \leq C(1+|x|)^{-\rho}$ then one has in $\mathcal{B}(\H_{spec})$ the equality
\begin{align*}
F_0(W_--\textup{Id} )F_0^* &= -2\pi i G_-( \psi(D_+) \otimes \textup{Id}_{H^{0,\rho-t}} )B.
\end{align*}
\end{thm}
\begin{proof}
Take $f \in C_c(\R^+,\mathcal{P})$ and $g \in C_c^\infty(\R^+) \odot C(\Sf^{n-1})$, and set $m = \rho - t$. We will show that $\langle F_0 (W_--\textup{Id} ) F_0^* f, g \rangle_{\H_{spec}} $ is equal to $\left\langle -2\pi i N(\psi(D_+) \otimes \textup{Id}_{H^{0,m}} ) B f, g \right\rangle_{\H_{spec}}$. We write $\chi_+$ for the characteristic function of $\R^+$ and note that $q_\lambda$ has compact support to obtain the following (in the sense of distributions with values in $H^{0,-m}$).
\begin{align*}
\int_0^\infty{\int_\R{\e^{i\nu z} q_\lambda(\nu)\, \d \nu}\, \d z} &= (2\pi)^\frac12 \int_\R{[\Fi^* \chi_+](\nu) q_\lambda(\nu)\, \d \nu} \\
&= (2\pi)^\frac12 \int_{-\lambda}^\infty{[\Fi^* \chi_+](\nu) \left(\frac{\nu+\lambda}{\lambda} \right)^\frac{1}{4} p(\nu+\lambda)\, \d \nu} \\
&= (2\pi)^\frac12 \int_\R{ [\Fi^* \chi_+](\lambda(\e^\mu-1)) \lambda \e^{\frac{5\mu}{4}} p(\e^\mu \lambda)\, \d \mu} \\
&= (2\pi)^\frac12 \int_\R{[\Fi^* \chi_+](\lambda(\e^\mu-1)) \lambda \e^{\frac{3\mu}{4}} [(U_+(\mu) \otimes \textup{Id}_{H^{0,-m}}) p](\lambda)\, \d \mu}.
\end{align*}
Next we note that the inverse Fourier transform of the characteristic function $\chi_+$ is the distribution $[\Fi^* \chi_+](y) = \pi^\frac12 2^{-\frac12} \delta(y) + i(2\pi)^{-\frac12} \textup{Pv}\frac{1}{y}$ (here $\textup{Pv}$ denotes the principal value), from which we obtain
\begin{align*}
\int_0^\infty{\int_\R{\e^{i\nu z} q_\lambda(\nu)\, \d \nu}\, \d z} &= \int_\R{\left(\pi \delta(\e^\mu-1)+i \textup{Pv} \frac{\e^{\frac{3\mu}{4}}}{\e^\mu-1} \right)[(U_+(\mu) \otimes \textup{Id}_{H^{0,-s}} ) p](\lambda)\Bigg\rangle_{H^{0,m},H^{0,-m}}\, \d \mu}
\end{align*}
Next we note that
\begin{align*}
\frac{\e^{\frac{3\mu}{4}}}{\e^\mu-1} &= \frac{1}{4} \left( \frac{1}{\sinh{\left(\frac{\mu}{4} \right)}} + \frac{1}{\cosh{\left(\frac{\mu}{4}\right)}} \right)
\end{align*}
and the known Fourier transform of the function $\psi$ of Equation \ref{eq:dilationfunc} in \cite[Table 20.1]{jeffrey08} as
\begin{align*}
[\Fi \psi](\mu) &= \pi^\frac12 2^{-\frac12} \delta(\e^\mu-1) + \frac{i}{4} (2\pi)^{-\frac12} \textup{Pv} \left( \frac{1}{\sinh{\left(\frac{\mu}{4} \right)}} + \frac{1}{\cosh{\left(\frac{\mu}{4}\right)}} \right).
\end{align*}
Combining these facts, we obtain
\begin{align*}
\langle F_0(W_--\textup{Id}  f, g \rangle_{\H_{spec}} 
= i \int_{\R^+}\Bigg\langle h_0(\lambda),&\int_\R \Bigg(\pi \delta(\e^\mu-1) +\frac{1}{4} \textup{Pv} \Bigg( \frac{1}{\sinh{\left(\frac{\mu}{4} \right)}} + \frac{1}{\cosh{\left(\frac{\mu}{4}\right)}}  \Bigg) \\
&\times  [(U_+(\mu) \otimes \textup{Id}_{H^{0,-m}} )p ](\lambda)\Bigg) \, \d \mu \Bigg\rangle_{H^{0,m},H^{0,-m}}\, \d \lambda.
\end{align*}
We now note that $[Bf](\lambda) = h_0(\lambda)$, $p = G_-^* g$ and
\begin{align*}
(\psi(D_+) \otimes \textup{Id}_{H^{0,m}})^* p &= (2\pi)^{-\frac12} \int_\R{[\Fi \psi](\mu) (U_+(\mu)\otimes \textup{Id}_{H^{0,m}}) p\, \d \mu}.
\end{align*}
Combining these facts we obtain 
\begin{align*}
\langle F_0(W_--\textup{Id} ) F_0^* f, g \rangle_{\H_{spec}} &= 2\pi i \int_{\R^+}{\left\langle [B f](\lambda), [(\psi(D_+)^* \otimes \textup{Id}_{H^{0,-s}})G_-^* g](\lambda)\right\rangle_{H^{0,m},H^{0,-m}}\, \d \lambda} \\
&= \left\langle -2\pi iG_- (\psi(D_+) \otimes \textup{Id}_{H^{0,m}}) Bf, g \right\rangle_{\H_{spec}}.
\end{align*}
Since both $C_c(\R^+, \mathcal{P})$ and $C_c^\infty(\R^+) \odot C(\Sf^{n-1})$ are dense in $\H_{spec}$, this concludes the proof.
\end{proof}

\begin{lemma}\label{lem:commutator}
Take $s \geq 0$ and $t > \frac{n}{2}$, and suppose that there are no resonances for $n = 4$. Then the difference
\begin{align}\label{eq:compactdiff}
&(\psi(D_+)\otimes \textup{Id}_{\mathcal{P}}) G_- -G_- (\psi(D_+) \otimes \textup{Id}_{H^{s,t}})
\end{align}
is in $\mathcal{K}(L^2(\R^+, H^{s,t}), \H_{spec})$.
\end{lemma}
This is a straightforward generalisation of \cite[Lemma 2.7]{richard13} and thus we omit the proof.

We now recall the action of the dilation group in $\R^n$ as $[U_n(t)f](x) = \e^{\frac{nt}{2}} f(\e^t x)$, and denote its self-adjoint generator by $D_n$. Then we have the relation $F_0 D_n F_0^* = -2 D_+ \otimes \textup{Id}_{\mathcal{P}}$ on $\Dom( D_+ \otimes \textup{Id}_{\mathcal{P}})$. Thus if we define $\vp(x) = \frac12(1+\tanh{(\pi x)} -i \cosh{(\pi x)}^{-1})$, we find
\begin{align}\label{eq:dilationconj}
F_0 \vp(D_n) F_0^* &= \psi(D_+) \otimes \textup{Id}_\mathcal{P}.
\end{align}
We can now prove the main result.
\begin{proof}[Proof of Theorem \ref{thm:main}]
We deduce from Theorem \ref{thm:hardwork} and Lemma \ref{lem:commutator} that
\begin{align*}
W_- - \textup{Id} &= -2\pi i F_0^*G_-(\psi(D_+)\otimes \textup{Id}_{H^{0,\rho-t}}) B F_0 \\
&= -2\pi i F_0^*\big[G_-(\psi(D_+)\otimes \textup{Id}_{H^{0,\rho-t}}) -(\psi(D_+) \otimes \textup{Id}_\mathcal{P})G_-+(\psi(D_+) \otimes \textup{Id}_\mathcal{P})G_- \big] B F_0 \\
&= -2\pi i F_0^*(\psi(D_+) \otimes \textup{Id}_{\mathcal{P}}) G_-B F_0 + K \\
&= \vp(D_n) F_0^* (-2\pi i) G_-B F_0+K,
\end{align*}
with $K:= -2\pi i F_0^*\big(G_-(\psi(D_+) \otimes \textup{Id}_{H^{0,\rho-t}})-(\psi(D_+) \otimes \textup{Id}_\mathcal{P})G_- \big)B F_0 \in \mathcal{K}(\H)$. Note that the last equality follows from Equation \eqref{eq:dilationconj}.
We can compute explicitly that
\begin{align*}
-2\pi i [G_-B f](\lambda,\omega) &= -2\pi i [\Gamma_0(\lambda) T(\lambda+i0) \Gamma_0(\lambda)^* f(\lambda,\cdot)](\omega) 
= [(S(\lambda)-\textup{Id})f(\lambda,\cdot)](\omega),
\end{align*}
by Theorem \ref{thm: stationary scattering operator}.
Thus we see that $F_0^*(-2\pi i) G_-B F_0 = S-\textup{Id}$, proving the claim.
\end{proof}

We note that the case $n = 4$ differs from the case $n = 2$ in the contribution of different types of resonances. As shown in \cite{richard21} in dimension $n = 2$, we need to exclude $p$-resonances but not $s$-resonances to obtain our structural formula. This distinction comes from a detailed analysis of the low energy behaviour of the operator $T(z)$ as in \cite{jensen01} for dimension $n = 2$ and \cite{jensen84} for dimension $n = 4$. When computing a low energy expansion for $T(z)$ in dimension four, all `bad' behaviour in the expansion is confined to a single rank-one operator, whilst in dimension two all `bad' behaviour in the expansion is confined to a rank-one and a rank-two operator which allows us to distinguish the different kinds of resonances. In dimensions $n=2,4$ it is known that resonances give an integer contribution to statements of Levinson's theorem \cite[Theorem 6.3]{bolle88} in contrast to the half-integer contributions which occur in dimensions $n=1,3$.

We also note that as an immediate consequence of Theorem \ref{thm:main} we can view Levinson's theorem in all dimensions as an index theorem in an identical manner to that of the case $n = 3$ discussed in \cite[Sections 4-6]{kellendonk12}, to which we refer for details. In the next section, we instead use the result of Theorem \ref{thm:main} to give a new topological interpretation of Levinson's theorem.

%%%%%%%%%%%%%%%%%%%%%%%%%%%%%%%%%%%%%%%%%%%%%%%%%%%%%%%%%%%%%%%%%%%%%
%%%%%%%%%%%%%%%%%%%%%%%%%%%%%%%%%%%%%%%%%%%%%%%%%%%%%%%%%%%%%%%%%%%%%
%%%%%%%%%%%%%%%%%%%%%%%%%%%%%%%%%%%%%%%%%%%%%%%%%%%%%%%%%%%%%%%%%%%%%

\section{Levinson's theorem as an index pairing}\label{sec:pairing}

We now show how the formula for the wave operator implies that the number of bound states is the result of an index pairing between $K$-theory and $K$-homology, see \cite{CPR4, HR}. We first recall the definition of a spectral triple.

\begin{defn} 
\label{def:ST}
An odd
spectral triple $(\A,\H,\D)$ is given by a Hilbert space $\H$, a
$*$-subalgebra  $\A\subset\B(\H)$ acting on
$\H$, and a densely defined unbounded self-adjoint operator $\D$  such that:

1. $a\cdot{\rm dom}\,\D\subset {\rm dom}\,\D$ for all $a\in\A$, so that
$da:=[\D,a]$ is densely defined.  Moreover, $da$ extends to a bounded operator  for all $a\in\A$;

2. $a(1+\D^2)^{-1/2}\in\K(\H)$ for all $a\in\A$.
\end{defn}

Spectral triples define classes in the $K$-homology of the 
norm closure $\overline{\A}$, a $C^*$-algebra. We will produce a spectral triple for $C_c^\infty((0,\infty),\K)$ where $\K$ is the compact operators on $L^2(\Sf^{n-1})$.

\begin{lemma}
The spaces $L^2(\R,\d x)$ and $L^2(\R^+, \d y)$ are unitarily equivalent via the map $W: L^2(\R) \to L^2(\R^+)$ given by 
\begin{align}
[Wf](y) &= y^{-\frac12} f(\ln(y))
\end{align}
with adjoint $W^*:L^2(\R^+, \d y) \to L^2(\R,\d x)$ given by $
[W^*g](x) = \e^{\frac{x}{2}} g(\e^x).
$
\end{lemma}
\begin{proof}
This is a simple check.
%We could show directly that $W$ is the desired operator by simply checking unitarity, however it is more useful to see how $W$ is constructed. First, we define the operator $U:L^2(\R,\d x) \to L^2\left( \R^+, \frac{\d \lambda}{\lambda} \right)$ via
%\begin{align}
%[Uf](\lambda) &= f(\ln(\lambda)).
%\end{align}
%A quick calculation shows that $U$ is unitary
%\begin{align*}
%\langle Uf, Ug \rangle &= \int_{\R^+}{ \overline{f(\ln(\lambda))} g(\ln(\lambda))\, \frac{\d \lambda}{\lambda}} \\
%			  &= \int_\R{\overline{f(x)} g(x) \, \d x} \\
%			  &= \langle f, g \rangle,
%\end{align*}
%where we have made the substitution $\lambda = \e^x$ and the measure then transforms as $\d\lambda = \e^x \d x$. 
%with inverse  $[U^* g](x) = g(\e^x)$. To change to  Lebesgue measure on the half-line we define the transformation $V: L^2\left( \R^+, \frac{\d \lambda}{\lambda} \right) \to L^2(\R^+, \d t)$ by
%\begin{align}
%[V h](y) = y^{-\frac12} h(y).
%\end{align}
%Again unitarity is a simple check
%\begin{align*}
%\langle Vh, Vg \rangle &= \int_{\R^+}{   \overline{y^{-\frac12}h(y)} y^{-\frac12} g(y)\, \d y} \\
%			  &= \int_{\R^+}{\overline{ h(y)} g(y) \frac{\d y}{y}} \\
%			  &= \langle h, g \rangle.
%\end{align*}
%Note that 
%and the inverse is  $[V^* g](\lambda) = \lambda^\frac12 g(\lambda)$. We can then define our unitary equivalence between $L^2(\R, \d x)$ and $L^2(\R^+, \d y)$ via the composition $W = VU$, which is unitary by construction.
\end{proof}

\begin{lemma}
\label{lem:arr-plus}
The two spectral triples
$$
\Big(C_c^\infty(\R),L^2(\R,\d x),\frac{1}{i}\frac{\d}{\d x}\Big)
\quad\mbox{and}\quad
\Big(C_c^\infty(\R^+),L^2(\R^+,\d y),\frac{y}{i}\frac{\d}{\d y}+\frac{1}{2i}\Big)
$$
are unitarily equivalent.
\end{lemma}
\begin{proof}
Conjugating $\frac{\lambda}{i}\frac{\d}{\d\lambda}+\frac{1}{2i}$ by $W: L^2(\R) \to L^2(\R^+)$ gives $\frac{1}{i}\frac{\d}{\d x}$.
\end{proof}
Since $-i\od{}{x}$ defines a non-trivial (indeed generating) $K$-homology class for $C_0(\R)$, we see that the generator of dilations on the half-line defines a non-trivial $K$-homology class for $C_0(\R^+)$.

Using this identification we immediately get the following in our context.
\begin{cor}
\label{cor:spectrip}
The data $(C_c^\infty(\R^+) \otimes \mathcal{K}(L^2(\Sf^{n-1})), L^2(\R^+) \otimes L^2(\Sf^{n-1}), D_+ \otimes \textup{Id})$ defines an odd spectral triple, and so a class $[D_+]$ in odd $K$-homology.
\end{cor}

\begin{lemma}
\label{lem:ess-kay}
When $S(0)=\textup{Id}$ the scattering operator defines an element of the odd $K$-theory $[S]\in K_1(C_0(\R^+)\ox\K)$. For suitably decaying potentials, for example those satisfying Assumption \ref{ass:best-ass}, $S(0)=\textup{Id}$ in all dimensions except dimension 1 (generically) and dimension 3 (in the presence of resonances).
%This happens when $n =2$, $n = 3$ with no resonances or $n \geq 4$ (need an efficient way to describe $\rho$ here...).
\end{lemma}

Corollary \ref{cor:spectrip} and Lemma \ref{lem:ess-kay} tell us that we can pair the classes $[D_+]$ and $[S]$ to obtain an integer. See \cite[Section 8.7]{HR} for details. Sadly the fact that our algebra $C_0(\R^+)\ox\K$ is nonunital (both $C_0(\R^+)$ and $\K$!) means that we need to be careful about applying index pairing formulae. See \cite[Section 2.3]{CGRS2} for a description of the problems.

In \cite[Section 2.7]{CGRS2} it was shown that for spectral triples $(\A,\H,\D)$ over a nonunital algebra, the index pairing with the class of a unitary $u$ can be computed as ${\rm Index}(PuP-(\textup{Id}-P))$ where $P=\chi_{[0,\infty)}(\D)$ is the non-negative spectral projection of $\D$. 

In our setting, however, we will be using an approximation of the non-positive spectral projection of $D_+ \otimes \textup{Id}$ and we need to show that we still get a Fredholm operator.

\begin{lemma}
\label{lem:technical}
If $B = GG^*$ is invertible modulo compacts then $G$ is Fredholm.
\end{lemma}
\begin{proof}
Write $B = A+K$ where $A$ is invertible and $K$ is compact. Then $GG^*A^{-1} = BA^{-1} = \textup{Id}+KA^{-1}$ and so $G$ is Fredholm with approximate right inverse $G^*A^{-1}$.
\end{proof}

\begin{thm}\label{thm:pairing}
Let $H = H_0+V$ be such that the wave operators exist, are complete and are of the form of Equation \eqref{eq:waveopform} with $S(0) = \textup{Id}$. Let $S$ be the corresponding scattering operator. We have the pairing
\begin{align*}
\langle [S], [D_+] \rangle &= -\textup{Index}(W_-) = N,
\end{align*}
where $N$ is the number of bound states of $H$ (eigenvalues counted with multiplicity).
\end{thm}
\begin{proof}
%We write
%\begin{align*}
%W_- &= \textup{Id}+\vp(D_n)(S-\textup{Id})+K \\
%&= \textup{Id}- \vp(D_n) + \vp(D_n)S + K \\
%&= \textup{Id}-\vp(D_n)+\vp(D_n)S\vp(D_n) +\vp(D_n)S(\textup{Id}-\vp(D_n))+K \\
%&= \textup{Id} - \vp(D_n) + \vp(D_n) S\vp(D_n) + \vp(D_n)(\textup{Id}-\vp(D_n))S + \vp(D_n) [S,\vp(D_n)] + K.
%\end{align*}
%We know that $[S,\vp(D_n)]$ is compact and thus we only need to show that the term $\vp(D_n)(\textup{Id}-\vp(D_n))S$ will make no contribution to the index. This follows from Lemma \ref{lem:technical}.
%
We know from \cite[Section 8.7]{HR} that the pairing can be computed as 
\begin{align*}
\langle [S], [ D_+] \rangle &=  \textup{Index}(PSP-(\textup{Id}-P)) = \textup{Index}(PSP+(\textup{Id}-P)),
\end{align*}
where $P$ is the non-negative spectral projection for $D_+ \otimes \textup{Id}$. The second equality follows from the homotopy $t \mapsto \e^{i \pi (1-t)} (\textup{Id}-P)$. We note that the pairing can also be computed using the non-positive spectral projection $P_-$ for $D_+\otimes \textup{Id}$, at the expense of a minus sign, since $[2P_+-1]=-[2P_--1]\in K^1(C_0(\R^+,\K))$. The result is 
\begin{align*}
\langle [S], [ D_+] \rangle &=  -\textup{Index}(P_-SP_--(\textup{Id}-P_-)) = -\textup{Index}(P_-SP_-+(\textup{Id}-P_-)).
\end{align*}
We will use this second form, and for convenience drop the subscript `$-$'.
To compute the pairing, we need to consider the wave operator in the spectral representation. We recall that the operator $D_n$ satisfies $F_0 D_n F_0^* = - 2 (D_+ \otimes \textup{Id})$.

So we need to be able to approximate $P$ with the operator $\vp(-2 D_+) \otimes \textup{Id}$. We initially work with $D_n$ for convenience. We set $T = \vp(D_n)$ and $t = \tanh{(\pi D_n)}$ for simplicity. A quick computation shows that
\begin{align*}
TT^* &= T^*T = \frac12(\textup{Id}+t).
\end{align*}
%and so (modulo compacts) we have
%\begin{align*}
%(\textup{Id}-T+TST^*)(\textup{Id}-T^*+T^*S^*T) &= \textup{Id}+\frac{1}{4} \cosh{(\pi D_n)}^{-2}
%\end{align*}
%which is invertible and so $\textup{Id}-T+TST^*$ is Fredholm.
Starting from $W_-$ there is a compact operator $K$ such that
\begin{align*}
W_- &= \textup{Id}-\frac12 (\textup{Id}+t)+TST^*+K.
\end{align*}
We can also find a compact $K_1$ such that
\begin{align*}
\frac12 (\textup{Id}-t) TS^*T^* &= \frac12 (\textup{Id}-t) TT^* S^* + \frac12 (\textup{Id}-t) [T,S^*]T^* \\
&= \frac{1}{4} (\textup{Id}-t^2)S^* + \frac12 (\textup{Id}-t) [T,S^*]T^* \\
&= \frac{1}{4}(\textup{Id}-t^2)+\frac{1}{4}(\textup{Id}-t^2)(S^*-\textup{Id}) +\frac12 (\textup{Id}-t) [T,S^*]T^* \\
&:= \frac{1}{4}(\textup{Id}-t^2) + K_1.
\end{align*}
Similarly we find compacts $K_2,K_3$ such that
\begin{align*}
\frac12 TST^*(\textup{Id}-t) &= \frac{1}{4}(\textup{Id}-t^2) +K_2, \\
TST^*TS^*T^* &= \frac{1}{4} (\textup{Id}+t)^2 +K_3.
\end{align*}

We can then check (with all equalities mod compacts) that 
\begin{align*}
W_-W_-^*=(\textup{Id}-\frac12 (\textup{Id}+t)+TST^*)(\textup{Id}-\frac12 (\textup{Id}+t)+TS^*T^*) &= \textup{Id}- \frac12 (\textup{Id}-t^2)
\end{align*}
which is invertible and hence $\textup{Id}-\frac12 (\textup{Id}+t)+TST^*$ is invertible modulo compacts by Lemma \ref{lem:technical}. Now we compare $\textup{Id}-\frac12 (\textup{Id}+t)+TST^*$ and $PSP+(\textup{Id}-P)$ to see that they have the same Fredholm index.
Taking commutators we have (mod compacts) the equalities
\begin{align*}
\textup{Id}-\frac12 (\textup{Id}+t)+TST^*
&=\textup{Id}-TT^*+TST^*=\textup{Id}+T(S-\textup{Id})T^*
\end{align*}
and
\[
PSP+(\textup{Id}-P)=\textup{Id}+P(S-\textup{Id})P.
\]
Define in the spectral representations the operators $\tilde{T}$ and $\tilde{t}$ by $\tilde{T} = F_0 T F_0^* = \frac12(\textup{Id}-\tanh{(2\pi D_+)} - i\cosh{(2\pi D_+)}^{-1}) \otimes \textup{Id}$ and $\tilde{t} = F_0 t F_0^* = -\tanh{(2\pi D_+)} \otimes \textup{Id}$. Taking the difference, and using $P=\frac12(\textup{Id}-\textup{sign}(D_+))$ we thus have
\begin{align}
&(P \otimes \textup{Id})S(P \otimes \textup{Id})+(\textup{Id}-(P \otimes \textup{Id}))-\big(\textup{Id}-\tilde{T}\tilde{T}^*+\tilde{T}S\tilde{T}^*\big)\nonumber \\
&=P(S-\textup{Id})P-\tilde{T}(S-\textup{Id})\tilde{T}^*\nonumber\\
&=(P-\tilde{T}\tilde{T}^*)(S-\textup{Id})\bmod\mbox{compacts}\nonumber\\
&=\frac12(-\textup{sign}+\tanh)(D_+)(S-\textup{Id})\bmod\mbox{compacts}.
\label{eq:difference}
\end{align}
We note that the second equality, where we have commuted the exact spectral projection with $S$ up to compacts, is justified by the arguments in \cite[Section 2.7]{CGRS2}.

To complete the argument, observe that $\frac12(-\textup{sign}+\tanh)$ is an $L^2$ function, and for any compactly supported function $\chi\in C_c(\R^+)$ we have $\lambda\mapsto \chi(\lambda)\Vert S(\lambda)-\textup{Id}\Vert_{\B(\mathcal{P})}$ is $L^2$ as well. Hence
\begin{align*}
&\frac12(-\textup{sign}+\tanh)(D_+)\chi(S(\cdot)-\textup{Id})\\
&=\frac12(-\textup{sign}+\tanh)(D_+)\chi\Vert(S(\cdot)-\textup{Id})\Vert_{\mathcal{B}(\mathcal{P})} \Vert(S(\cdot)-\textup{Id})\Vert_{\mathcal{B}(\mathcal{P})}^{-1}(S(\cdot)-\textup{Id})
\end{align*}
is a product of $L^2$ functions, one of $D_+$ and one of $H_0$, composed with a uniformly bounded compact operator-valued function $\Vert(S(\cdot)-\textup{Id})\Vert_{\mathcal{B}(\mathcal{P})}^{-1}(S(\cdot)-\textup{Id})$. Applying the standard $f(x)g(\nabla)$ result 
\cite[Theorem 4.1]{simon79} we obtain a compact operator times a uniformly bounded operator, which is then compact.
As $\lambda \mapsto \Vert(S(\lambda)-\textup{Id})\Vert_{\B(\mathcal{P})}$ is continuous and vanishes at $\lambda=0,\infty$, we can take an approximate unit $\chi_m\in C_c(\R^+)$ and see that 
\[
\frac12(-\textup{sign}+\tanh)(D_+)\chi_m\Vert(S(\cdot)-\textup{Id})\Vert_{\mathcal{B}(\mathcal{P})}
\]
converges in norm. Hence the limit $\frac12(-\textup{sign}+\tanh)(D_+)\Vert(S(\cdot)-\textup{Id})\Vert_{\mathcal{B}(\mathcal{P})}$ is compact.

%  The function $(\textup{sign}-\tanh)(D_+)$ is square integrable and thus by Lemma \ref{lem:f(x)g(nabla)} (see also \cite[Theorem 4.1]{simon79}) the difference is compact.
Since the difference \eqref{eq:difference} is compact, the two Fredholm operators $W_-$ and $P(S-\textup{Id})P$ have the same index.
The fact that $\textup{Ind}(W_-) = - N$ follows immediately from the relations
\begin{align*}
W_-W_-^* &= \textup{Id} \quad \textup{ and } \quad W_-^*W_- = \textup{Id}-P_p(H),
\end{align*}
where $P_p(H)$ denotes the projection onto the point spectrum of $H$. 
\end{proof}

\begin{cor}
Let $V_0$ and $V_1$ be such that the wave operators exist, are complete and of the form of Equation \eqref{eq:waveopform}. Suppose further that $S_0(0) = S_1(0) = \textup{Id}$, where $S_0$ and $S_1$ denote the scattering operators for the potentials $V_0$ and $V_1$. If the number of eigenstates (counted with multiplicity) for $V_0$ differs from that of $V_1$ then their scattering matrices are not (stably) homotopic. 
\end{cor}

\begin{cor}\label{cor:norm-cont}
Let $V_0$ and $V_1$ be such that the wave operators exist, are complete and of the form of Equation \eqref{eq:waveopform}. Suppose further that $S_0(0) = S_1(0) = \textup{Id}$, where $S_0$ and $S_1$ denote the scattering operators for the potentials $V_0$ and $V_1$. Consider the path $V_t = (1-t)V_0 + t V_1$ for $t \in [0,1]$ with corresponding scattering operators $S_t$. If the number of eigenstates (counted with multiplicity) for $V_0$ differs from that of $V_1$, then the path $S_t$ is not norm continuous.
\end{cor}
In fact the proof shows that there are a discrete number of points at which the path $S_t$ fails to be norm continuous, corresponding to `jumps' in the number of eigenvalues for the potential $V_t$. The norm holomorphy (which implies norm continuity) of the scattering operator as a function of $t$ is discussed in \cite[Theorem 4.2]{bruning82} where an equivalent condition to holomorphy of $S_t$ is given. The points of failure of norm continuity in Corollary \ref{cor:norm-cont} are also points of failure of holomorphy in \cite[Theorem 4.2]{bruning82}. In the case of a Rollnik class potential on $\R^3$ holomorphy of the scattering matrix as a function of $t$ has been studied in \cite[Theorem 6.1]{kato66} and \cite[Theorem 5.2]{bruning82}.

The case $S(0) \neq \textup{Id}$ cannot be handled by the simple pairing described in Theorem \ref{thm:pairing}, a more complicated object than simply the dilation operator is required to pair with the class of the scattering matrix. Such results will be discussed elsewhere.

\end{document}